\documentclass[11pt]{amsart}

\usepackage{amsmath,latexsym,amssymb,amsthm,array,amsfonts}
\usepackage{hyperref}
\usepackage{mathrsfs,dsfont,bbm}

\newtheorem{theorem}{Theorem}[section]
\newtheorem{lemma}{Lemma}[section]
\newtheorem{corollary}{Corollary}[section]

\newtheorem{proposition}[theorem]{Proposition}

\newtheorem{definition}{Definition}[section]

\numberwithin{equation}{section}

\begin{document}

\title[Derivative for the intersection local time of fBMs]
{Derivative for the intersection local time of fractional Brownian
Motions${}^{*}$}

\footnote[0]{${}^{*}$The Project-sponsored by NSFC (11171062) and Innovation Program of Shanghai Municipal Education Commission(12ZZ063).}


\author[L. Yan]{Litan Yan}



\maketitle

\date{}

\begin{center}
{\footnotesize {\it {Department of Mathematics, College of Science, Donghua University\\
2999 North Renmin Rd., Songjiang, Shanghai 201620, P.R. China}}\\
E-mail : litanyan@hotmail.com}
\end{center}

\maketitle

\begin{abstract}
Let $B^{H_1}$ and $\tilde{B}^{H_2}$ be two independent fractional Brownian motions on ${\mathbb R}$ with respective indices $H_i\in (0,1)$ and $H_1\leq H_2$. In this paper, we consider their intersection local time $\ell_t(a)$. We show that $\ell_t(a)$ is differentiable in the spatial variable if $\frac1{H_1}+\frac1{H_2}>3$, and we introduce the so-called {\it hybrid quadratic covariation} $[f(B^{H_1}-\tilde{B}^{H_2}),B^{H_1}]^{(HC)}$. When $H_1<\frac12$, we construct a Banach space ${\mathscr H}$ of measurable functions such that the quadratic covariation exists in $L^2(\Omega)$ for all $f\in {\mathscr H}$, and the Bouleau-Yor type identity
$$
[f(B^{H_1}-\tilde{B}^{H_2}),B^{H_1}]^{(HC)}_t=-\int_{\mathbb R}f(a)\ell_t(da)
$$
holds. When $H_1\geq \frac12$, we show that the quadratic covariation exists also in $L^2(\Omega)$ and the above Bouleau-Yor type identity holds also for all H\"older functions $f$ of order $\nu>\frac{2H_1-1}{H_1}$.
\end{abstract}


\section{Introduction}
In the study of stochastic area integrals for Brownian motion
$B$, Rogers-Walsh~\cite{Rogers1,Rogers2} were led to analyze the
following functional
$$
A(t,x)=\int_0^t1_{[0,\infty)}(x-B_s)ds,\quad t\geq 0,\;x\in {\mathbb R}.
$$
By using the classical It\^o calculus they showed that the process $\{A(t,B_t),t\geq 0\}$ is not a semi-martingale. In fact, they showed that the process
\begin{align}\label{eq1.4}
A(t,B_t)-\int_0^t{\mathscr L}(s,{B_s})dB_s
\end{align}
has finite non-zero $4/3$--variation, where ${\mathscr
L}(t,x)=\int_0^t\delta(B_s-x)ds$ is the local time. Now, a natural idea is to consider the functional
$$
A(t,\widetilde{B}_t)=\int_0^t1_{[0,\infty)}(\widetilde{B}_t-B_s)ds,\quad t\geq 0,
$$
where $\widetilde{B}$ is Brownian motion of independent $B$.

On the other hand, by a formal application of It\^o's formula with respect to the Brownian motion and using
$$
\frac{d}{dx}1_{\{x\geq 0\}}=\delta(x),\qquad \frac{d^2}{dx^2}1_{\{x\geq 0\}}=\delta'(x)
$$
in the sense of Schwartz's distribution, Rosen~\cite{wode19} developed a new approach to the study of $A(t,B_t)$ as follows:
$$
A(t,B_t)-\int_0^t{\mathscr
L}(s,B_s)dB_s=t+\frac12\int_0^t\int_0^s\delta'(B_s-B_r)drds
$$
for all $t\geq 0$, and one can consider the process
$$
\alpha_t'(a):=-\int_0^t\int_0^s\delta'(B_s-B_r-a)drds, \quad t\geq 0,a\in {\mathbb R}
$$
which are called the {\em derivatives of self-intersection local time} (in short, DSLT) of Brownian motion. By using the idea, Yan et al.~\cite{Yan2} deduced the existence of process
$$
\beta'_t(a):=-\int_0^tds\int_0^s\delta'(B_{r}^{H}- B_{s}^{H}-a)dr, \quad t\geq 0,a\in {\mathbb R},
$$
which are called the DSLT of fractional Brownian motion (fBm) $B^H$. Moreover, Jung-Markowsky~\cite{Jung-Markowsky1,Jung-Markowsky2} considered some in-depth results for $\beta'_t(a)$. Motivated by these results, in this paper, as an extension we consider the so-called {\em derivatives of the intersection local time} (DILT) of fBms which is formally defined as follows
$$
\ell'_t(a):=-\int_0^tds\int_0^s\delta'(B_{r}^{H_1}- \tilde{B}_{s}^{H_2}-a)dr, \quad t\geq 0,a\in {\mathbb R},
$$
where $B^{H_1}$ and $\tilde{B}^{H_2}$ are two independent fBms with respective indices $H_i\in (0,1)$ on ${\mathbb R}$ and $H_1\leq H_2$.

This paper is organized as follows. In Section~\ref{sec2} we present some preliminaries for fBm. In Section~\ref{sec3} we find the exact result of the existence of the DILT $\ell'_t(0)$ and prove the H\"older continuity of the DILT $\ell'_t(a)$. As a corollary we have
$$
\ell'_t(a)=\frac{\partial}{\partial a}\ell_t(a),
$$
and the occupation formula
$$
\int_0^tds\int_0^sf'(B^{H_1}_r-\tilde{B}^{H_2}_s)dr=-\int_{\mathbb R}f(a)\ell'_t(a)da
$$
holds for any $f\in C^1({\mathbb R})$ and $t\in [0,T]$, provided either $\frac1{H_1}+\frac1{H_2}>3$. In Section~\ref{sec6} and~\ref{sec10}, we study the so-called {\it hybrid quadratic covariation}.
\begin{definition}
Let $0<H_1,H_2<1$ and let $f$ be a Borel function on ${\mathbb R}$ such that the following integral exists:
$$
J_\varepsilon(H_1,H_2,t,f):=\frac1{\varepsilon^{2H_1}}\int_0^tds\int_0^s \left\{\Delta_\varepsilon f(B^{H_1}_{r}-\tilde{B}^{H_2}_{s})\right\} \left(B^{H_1}_{r+\varepsilon}-B^{H_1}_r\right)dr
$$
for all $\varepsilon>0$, where
$$
\Delta_\varepsilon f(B^{H_1}_{r}-\tilde{B}^{H_2}_{s}):=f(B^{H_1}_{r+\varepsilon}-\tilde{B}^{H_2}_{s}) -f(B^{H_1}_{r}-\tilde{B}^{H_2}_{s}).
$$
The limit $\lim\limits_{\varepsilon\to 0}J_\varepsilon(H_1,H_2,t,f)$ is called the {\it hybrid quadratic covariation} (in short, HQC), provided the limit exists in $L^1(\Omega)$, denoted by $[f(B^{H_1}-\tilde{B}^{H_2}),B^{H_1}]^{(HC)}_t$.
\end{definition}
Clearly, we have
$$
[f(B^{H_1}-\tilde{B}^{H_2}),B^{H_1}]^{(HC)}_t= \int_0^tds\int_0^sf'(B^{H_1}_r-\tilde{B}^{H_2}_s)dr
$$
for all $0<H_1<1$ and $f\in C^1({\mathbb R})$. When $0<H_1<\frac12$, the HQC is considered in Section~\ref{sec6}. By considering the decomposition
\begin{equation}\label{sec1-eq1.007}
\begin{split}
J_\varepsilon(H_1,H_2,t,f)&=\frac1{\varepsilon^{2H_1}} \int_0^tds\int_0^s f(B^{H_1}_{r+\varepsilon}-\tilde{B}^{H_2}_s) \left(B^{H_1}_{r+\varepsilon}-B^{H_1}_r\right)dr\\
&\qquad\qquad-\frac1{\varepsilon^{2H_1}}\int_0^tds\int_0^s f(B^{H_1}_{r}-\tilde{B}^{H_2}_s) \left(B^{H_1}_{r+\varepsilon}-B^{H_1}_r\right)dr
\end{split}
\end{equation}
for all $\varepsilon>0$ and by estimating the two terms on the right hand side, respectively, we construct a Banach space ${\mathscr H}$ of measurable functions such that the HQC exists in $L^2(\Omega)$ for all $f\in {\mathscr H}$. Moreover, for all $f\in {\mathscr H}$ and $0<H_1<\frac12$ we show that the integral
$$
\int_{\mathbb R}f(a)\ell'_t(a)da
$$
is well-defined and the following Bouleau-Yor type identity holds:
\begin{equation}\label{sec1-eq1-100}
[f(B^{H_1}-\tilde{B}^{H_2}),B^{H_1}]^{(HC)}_t=-\int_{\mathbb R}f(a)\ell'_t(a)da\equiv-\int_{\mathbb R}f(a)\ell_t(da).
\end{equation}
When $H_1 \geq \frac12$, the HQC is considered in Section~\ref{sec10}. It is clear that the decomposition~\eqref{sec1-eq1.007} does not bring any information and we need a new idea for $H_1> \frac12$. In fact, for $f(x)=x$ we have
\begin{align*}
\frac1{\varepsilon^{2H_1}} \int_0^t ds\int_0^s
&E\left[(B^{H_1}_{r}-\tilde{B}^{H_2}_s)(B^{H_1}_{r+\varepsilon}-B^{H_1}_r) \right]dr\\
&=\frac1{\varepsilon^{2H_1}} \int_0^t ds\int_0^s
E\left[B^{H_1}_{r}(B^{H_1}_{r+\varepsilon}-B^{H_1}_r)\right]dr\\
&\longrightarrow \infty,
\end{align*}
as $\varepsilon\downarrow 0$. Thus, when $H_1\geq \frac12$, by estimating integrally the expression
$$
\frac1{\varepsilon^{2H_1}} \int_0^t ds\int_0^s\left\{\Delta_\varepsilon f(B^{H_1}_{r}-\tilde{B}^{H_2}_{s})\right\} \left(B^{H_1}_{r+\varepsilon}-B^{H_1}_r\right)dr,
$$
and by using the existence of the Young integral
$$
\int_{\mathbb R}f(a)\ell_t(da),
$$
we show that the HQC exists in $L^2(\Omega)$ and the Bouleau-Yor type identity~\eqref{sec1-eq1-100} holds for all H\"older functions $f$ of order $\nu>\frac{2H_1-1}{H_1}$. In Appendix we give the proofs of some basic estimates.

\section{Fractional Brownian motion}\label{sec2}
In this section, we briefly recall some basic results of fBm with $0<H<1$ and give some basic estimates. For more aspects on the material we refer to E. Al\'os {\em et al}~\cite{Nua1}, Biagini {\it et al}~\cite{BHOZ}, Cheridito-Nualart~\cite{Cheridito-Nualart}, Decreusefond-\"Ust\"unel~\cite{Dec}, Gradinaru {\em et al}~\cite{Grad1}, Hu~\cite{Hu3}, Mishura~\cite{Mishura2}, Nourdin~\cite{I. Nourdin}, Nualart~\cite{Nual} and references therein. parameters.

A zero mean Gaussian process $B^H=\{B_t^H, 0\leq t\leq T\}$ defined on $(\Omega, \mathcal{F}^H, P)$ is called the fBm with Hurst index $H\in (0,1)$ if $B^H_0=0$ and
$$
E\left[B^H_tB^H_s\right]=\frac{1}{2}\left[t^{2H}+s^{2H}-|t-s|^{2H}
\right]
$$
for $t,s\geq 0$. FBm $B^H$ admits the integral representation of the form
$$
B^H_t=\int_0^tK_H(t,s)dB_s,\qquad 0\leq t\leq T,
$$
where $B$ is a standard Brownian motion and the kernel $K_H(t,s)$
satisfies
$$
\frac{\partial K_H}{\partial t}(t,s)=\kappa_H\left(H-\frac12\right)
\left(\frac{s}{t}\right)^{\frac12-H}(t-s)^{H-\frac32}
$$
with a normalizing constant $\kappa_H>0$. Let $\mathcal H$ be the completion of the linear space ${\mathcal E}$ generated by the indicator functions ${1}_{[0,t]},
t\in [0,T]$ with respect to the inner product
$$
\langle {1}_{[0,s]},{1}_{[0,t]}
\rangle_{\mathcal H}=\frac{1}{2}\left[t^{2H}+s^{2H}-|t-s|^{2H}
\right].
$$
The application
$$
{\mathcal E}\ni \varphi\mapsto B^{H}(\varphi):=\int_0^T\varphi(s)dB^H_s
$$
is an isometry from ${\mathcal E}$ to the Gaussian space generated by
$B^{H}$ and it can be extended to ${\mathcal H}$. Denote by $\mathcal S$ the set of smooth functionals of the form
$$
F=f(B^{H}(\varphi_1),B^{H}(\varphi_2),\ldots,B^{H}(\varphi_n)),
$$
where $f\in C^{\infty}_b({\mathbb R}^n)$ ($f$ and all its
derivatives are bounded) and $\varphi_i\in {\mathcal H}$. The {\it
derivative operator} $D^{H}$ (the Malliavin derivative) of a
functional $F$ of the form above is defined as
$$
D^{H}F=\sum_{j=1}^n\frac{\partial f}{\partial
x_j}(B^{H}(\varphi_1),B^{H}(\varphi_2),
\ldots,B^{H}(\varphi_n))\varphi_j.
$$
The derivative operator $D^{H}$ is then a closable operator from
$L^2(\Omega)$ into $L^2(\Omega;{\mathcal H})$. We denote by
${\mathbb D}^{1,2}$ the closure of ${\mathcal S}$ with respect to
the norm
$$
\|F\|_{1,2}:=\sqrt{E|F|^2+E\|D^{H}F\|^2_{{\mathcal H}}}.
$$
The {\it divergence integral} $\delta^{H}$ is the adjoint of
derivative operator $D^{H}$. That is, we say that a random variable
$u$ in $L^2(\Omega;{\mathcal H})$ belongs to the domain of the
divergence operator $\delta^{H}$, denoted by ${\rm
{Dom}}(\delta^H)$, if
$$
E\left|\langle D^{H}F,u\rangle_{\mathcal H}\right|\leq
c\|F\|_{L^2(\Omega)}
$$
for every $F\in \mathcal S$. In this case $\delta^{H}(u)$ is defined by the duality relationship
\begin{equation}\label{sec2-eq2.1}
E\left[F\delta^{H}(u)\right]=E\langle D^{H}F,u\rangle_{\mathcal H}
\end{equation}
for any $u\in {\mathbb D}^{1,2}$. We have ${\mathbb D}^{1,2}\subset
{\rm {Dom}}(\delta^H)$. We will use the notation
$$
\delta^{H}(u)=\int_0^Tu_sdB^{H}_s
$$
to express the Skorohod integral of a process $u$, and the
indefinite Skorohod integral is defined as $\int_0^tu_sdB^{H}_s=\delta^H(u{1}_{[0,t]})$. We can localize the domains of the operators $D^H$ and $\delta^H$. If $\mathbb{L}$ is a class of random variables (or processes) we denote by $\mathbb{L}_{\rm loc}$ the set of random variables $F$ such that there exists a sequence $\{(\Omega_n, F^n), n\geq 1\}\subset {\mathscr F}^H\times \mathbb{L}$ with the following properties:

\hfill

(i) $\Omega_n\uparrow \Omega$, a.s.

(ii) $F=F^n$ a.s. on $\Omega_n$.

\hfill\\
If $F\in \mathbb{D}^{1,2}_{\rm loc}$, and $(\Omega_n, F^n)$
localizes $F$ in $\mathbb{D}^{1,2}$, then $D^HF$ is defined without
ambiguity by $D^HF=D^HF^n$ on $\Omega_n$, $n\geq 1$. Then, if $u\in
\mathbb{D}^{1,2}_{\rm loc}$, the divergence $\delta^H(u)$ is defined
as a random variable determined by the conditions
$$
\delta^H(u)|_{\Omega_n}=\delta^H(u^n)|_{\Omega_n}\qquad {\rm {
for\;\; all\;\;}} n\geq 1,
$$
where $(\Omega_n, u^n)$ is a localizing sequence for $u$, but it may depend on the localizing sequence.


\section{Existence and H\"older continuity of the DILT of fBms}\label{sec3}
In this section we will consider the existence and continuity of the DILT of fBms. Let $B^{H_{i}}=\left\{B^{H_{i}}_t,\, t\geq 0 \right\}$, $i=1,2$ be two independent fractional Brownian motions with respective indices $H_i\in (0,1)$and $H_1\leq H_2$. The intersection local time, denoted by $\ell_t(a)$, is formally defined by
$$
\ell_t(a)=\int_0^t ds\int_0^s\delta(B_r^{H_{1}}- \tilde{B}_{s}^{H_{2}}-a)dr, \qquad t\geq 0,\;a\in {\mathbb R},
$$
where $\delta$ denotes the Dirac delta function. Nualart and Ortiz-Latorre~\cite{Nualart and S. Ortiz-Latorre} has showed the random variables $\ell_t(0)$, $t\geq 0$ exist in $L^2$ (see also, Chen-Yan~\cite{Chen-Yan}, Jiang-Wang~\cite{J-W} and Wu-Xiao~\cite{Wu-Xiao}). By approximating the Dirac delta function by the heat kernel
\begin{equation}\label{sec3-eq3.2}
p_{\varepsilon}(x)=\frac{1}{\sqrt{2\pi\varepsilon}}e^{
-\frac{x^{2}}{2\varepsilon}}
\equiv\frac{1}{2\pi}\int_{\mathbb{R}}e^{ix\xi}e^{-\varepsilon
\frac{\xi^{2}}{2}}d\xi.
\end{equation}
with $\varepsilon>0$, one can define $\ell_t(a):=\lim_{\varepsilon\to 0}\ell_{\varepsilon,t}(a)$ in $L^2(\Omega)$, where
$$
\ell_{\varepsilon,t}(a)=\int_0^t ds\int_0^s p_{\varepsilon}(B_r^{H_{1}}-\tilde{B}_{s}^{H_{2}}-a)dr.
$$
Denote
\begin{equation}\label{sec3-eq3.3}
\begin{split}
\ell_{\varepsilon,t}'(a)&=-\int_0^t ds\int_0^s p'_{\varepsilon}(B_r^{H_{1}}-\tilde{B}_{s}^{H_{2}}-a)dr\\
&=-\frac{i}{2\pi}\int_0^t ds\int_0^s dr\int_{\mathbb{R}}\xi
e^{i\xi(B_r^{H_{1}}-\tilde{B}_{s}^{H_{2}}-a)}\cdot e^{-\varepsilon
\frac{\xi^{2}}{2}}d\xi.
\end{split}
\end{equation}
The process
$$
\ell'_t(a):=\lim_{\varepsilon\to 0}\ell'_{\varepsilon,t}(a)
$$
is called the {\em derivatives of the intersection local time} (DILT) of fBms, provided the limit exists in $L^1(\Omega)$. We first obtain the exact result of the existence.

For simplicity we assume that $C$ stands for a positive constant depending only on some determinate parameters, and moreover, the notation $F\asymp G$ means that the ratio $F(x)/G(x)$ is bounded from below and above by positive constants that do not depend on $x$ in the common domain of definition for $F$ and $G$. Denote
\begin{align*}
&\lambda_{r,s}:={\rm Var}(B^{H_1}_r-\tilde{B}^{H_2}_s)
=r^{2H_1}+s^{2H_2},\\
&\mu_{\varepsilon_1,\varepsilon_2}: =E\left[(B^{H_1}_{r+\varepsilon_1}-\tilde{B}^{H_2}_s) (B^{H_1}_{r'+\varepsilon_2}-\tilde{B}^{H_2}_{s'})\right]\\
&\rho^2_{\varepsilon_1,\varepsilon_2} =\lambda_{r+\varepsilon_1,s}\lambda_{r'+\varepsilon_2,s'} -\mu^2_{\varepsilon_1,\varepsilon_2}.
\end{align*}
for $s>r>0,s'>r'>0$ and $\varepsilon_1,\varepsilon_2>0$. Set $\mu:=\mu_{0.0}$ and $\rho:=\rho_{0,0}$. The next lemma will proved in Appendix~\ref{app2}
\begin{lemma}\label{lem2.5}
For all $s>r>0$ and $s'>r'>0$, we have
\begin{equation}\label{sec3-eq3.8-00}
\rho^2=\lambda_{r,s}\lambda_{r',s'} -\mu^2\asymp \left((r\wedge r')^{2H_{1}}+(s\wedge s')^{2H_{2}}\right)
\left(|r-r'|^{2H_1}+|s-s'|^{2H_2}\right).
\end{equation}
\end{lemma}
\begin{theorem}\label{th3.1}
For every $t>0$, $\ell_{\varepsilon,t}'(0)$ converges in $L^{2}(\Omega)$, as $\varepsilon$ tends to $0$ if $\frac1{H_1}+\frac1{H_2}>3$.
\end{theorem}
\begin{proof}
Denote ${\mathbb T}=\{0<r<s<t,\;0<r'<s'<t\}$ for any $t>0$. Then we have
\begin{align*}
E\ell_{\varepsilon,t}'(0)&=-\frac{i}{2\pi}\int_0^t ds\int_0^sdr \int_{\mathbb R}{\xi} Ee^{i\xi(B^{H_1}_r-\tilde{B}^{H_2}_s)}e^{-\varepsilon \xi^2/2}d\xi=0
\end{align*}
and
\begin{align*}
E\left[\ell_{\varepsilon,t}'(0)^2\right]&=\frac{-1}{(2\pi)^2}
\int_{\mathbb T} drdsdr'ds'\int_{{\mathbb R}^2}{\xi\eta}e^{-\frac{\varepsilon}2(\xi^2+\eta^2)}\\
&\qquad\qquad\cdot
E\exp\left(i(B^{H_1}_r-\tilde{B}^{H_2}_s)\xi+i(B^{H_1}_{r'}-\tilde{B}^{H_2}_{s'}) \eta\right)d\xi d\eta\\
&=\frac{-1}{(2\pi)^2} \int_{\mathbb T} drdsdr'ds'\int_{{\mathbb R}^2}{\xi\eta}\\
&\qquad\qquad\cdot
\exp\left(-\frac12\left[(\lambda_{r,s}+\varepsilon)\xi^2 +2\mu\xi\eta+(\lambda_{r',s'}+\varepsilon)\eta^2\right]\right)d\xi d\eta\\
&=\frac{1}{2\pi}\int_{\mathbb T}\frac{\mu drdsdr'ds' }{\left((\lambda_{r,s}+\varepsilon)(\lambda_{r',s'} +\varepsilon)-\mu^2\right)^{3/2}}
\end{align*}
for all $\varepsilon>0$, which deduce that $\ell_{\varepsilon,t}'(0)\in L^{2}(\Omega)$ if and only if
$$
\int_{\mathbb T}\frac{\mu drdsdr'ds'}{\left(\lambda_{r,s}\lambda_{r',s'} -\mu^2\right)^{3/2}}<\infty
$$
for all $t\geq 0$.

On the other hand, by Lemma~\ref{lem2.5} and Young's inequality, we get
\begin{equation}
\lambda_{r,s}\lambda_{r',s'}-\mu^2\geq C(r\wedge r')^{2\alpha H_{1}}(s\wedge s')^{2(1-\alpha)H_{2}}
|r-r'|^{2\alpha H_1}|s-s'|^{2(1-\alpha)H_2}
\end{equation}
for all $0\leq \alpha\leq 1$. Since $\frac1{H_1}+\frac1{H_2}>3$ we can take
$$
\alpha=\frac{H_2}{H_1+H_2}
$$
so that $3\alpha H_1=3(1-\alpha)H_2<1$. This proves
\begin{align*}
\int_{{\mathbb T}}\frac{|\mu| drdsdr'ds'}{\left(\lambda_{r,s}\lambda_{r',s'} -\mu^2\right)^{3/2}}<
\int_{{\mathbb T}}\frac{|\mu| drdsdr'ds'}{[(r\wedge r')(s\wedge s')
|r-r'||s-s'|]^\theta}<\infty
\end{align*}
with $\theta=\frac{3H_1H_2}{H_1+H_2}$ if $\frac1{H_1}+\frac1{H_2}>3$.

Finally, we claim that the sequence
$\{\ell_{\varepsilon,t}'(0),\varepsilon>0\}$ is of Cauchy in
$L^{2}(\Omega)$. For any $\delta, \varepsilon>0$ we have
\begin{equation*}
\begin{split}
E(|\ell_{\varepsilon,t}'(0)&-\ell_{\delta,t}'(0)|^{2})
=\frac{1}{4\pi^{2}}\int_{\mathbb T}drdsdr'ds'\int_{\mathbb{R}^{2}} \xi\eta
E\exp\left\{i\xi(B^{H_{1}}_r-\tilde{B}^{H_2}_s)
+i\eta(B^{H_1}_{r'}-\tilde{B}^{H_2}_{s'})\right\}\\
&\hspace{2cm}\cdot\left(e^{-\frac\varepsilon2\xi^2}-e^{
-\frac\delta2\xi^2}\right)
\left(e^{-\frac\varepsilon2\eta^2}-e^{-\frac\delta2\eta^2}\right)
d\xi d\eta drds\\
&=\frac{1}{4\pi^{2}}\int_{\mathbb T}drdsdr'ds'\int_{\mathbb{R}^{2}}
e^{-\frac{1}{2}(\lambda_{r,s}\xi^2+2\mu\xi\eta+\lambda_{r',s'}\eta^2)} \left(1-e^{-\frac{|\varepsilon-
\delta|^{2}|\xi|^{2}}{2}}\right)\\
&\hspace{2cm}\cdot\left(1-e^{-\frac{|\varepsilon-\delta|^{2}|\eta|^{2}}{2}}
\right) e^{-\frac{\varepsilon\wedge\delta}{2}(\xi^{2}+\eta^{2})}d\xi
d\eta.
\end{split}
\end{equation*}
Thus, dominated convergence theorem yields
$$
E(|\ell_{\varepsilon,t}'(0)-\ell_{\delta,t}'(0)|^{2})\longrightarrow 0
$$
as $\varepsilon\to 0$ and $\delta\to 0$, which leads to $
\{\ell_{\varepsilon,t}'(0),\;\varepsilon>0\}$ is a Cauchy sequence in $L^{2}(\Omega,\mathscr{F},P)$. Consequently, $\lim_{\varepsilon\to
0}\ell_{\varepsilon,t}'(0)$ exists in $L^{2}(\Omega, \mathscr{F},P)$. This completes the proof.
\end{proof}

\begin{corollary}
If $0<H_1<\frac1{3\alpha}$ and $0<H_2<\frac1{3(1-\alpha)}$ with $0<\alpha<1$, $\ell_{\varepsilon,t}'(0)$ converges in $L^{2}(\Omega)$ for every $t>0$, as $\varepsilon$ tends to $0$. In particular, if either $H_1<\frac12$ or $H_2<\frac23$, $\ell_{\varepsilon,t}'(0)$ converges in $L^{2}(\Omega)$ for every $t>0$, as $\varepsilon$ tends to $0$.
\end{corollary}


At the end of this section, we consider the H\"older continuity and the
occupation-time formula for the DILT $\{\ell^{'}_t(a); t\geq 0,a\in {\mathbb R}\}$. Our main object is to explain and prove the following theorem.

\begin{theorem}\label{th5.1}
Let $0<H_1\leq H_2<1$. If $\frac1{H_1}+\frac1{H_2}>3$, then the processes $\ell_{\varepsilon,t}'(a)$ converges almost surely, and in $L^p(\Omega)$ for all $p\in (0,\infty)$, as $\varepsilon$ tends to zero. Moreover, the process $\ell'_t(a)$ has a modification which is a.s. jointly H\"older continuous in $(a,t)$.
\end{theorem}

In order to prove the theorem we denote $\widetilde{{\mathbb D}}_{t,t'}(u)=\{t<u_n<u_{n-1}<\cdots<u_1<t'\}$ for $0\leq t<t'\leq T$, and
\begin{equation}\label{sec5-eq5.400-00}
\begin{split}
\widetilde{\Lambda}(t,t',n,\gamma):&=\int_{\widetilde{{\mathbb D}}_{t,t'}(u)}\int_{\widetilde{{\mathbb D}}_{0,t'}(v)}dv_1\cdots dv_ndu_1\cdots du_n\int_{\mathbb{R}^n} \prod_{j=1}^n|\xi_j|^{1+\gamma}d\xi_j\\
&\quad\cdot
\exp\left\{-\frac12 \kappa\left[(\sum\limits_{k=1}^n\xi'_k )^2(u_n)^{2H_1}+(\sum\limits_{k=1}^n\xi''_k )^2(v_n)^{2H_2}\right]\right\}\\
&\quad\cdot\prod_{j=1}^{n-1}\exp\left\{-\frac12 \kappa\left[( \sum\limits_{k=1}^{j}\xi'_k)^2(u_j-u_{j+1})^{2H_1}+ (\sum\limits_{k=1}^{j}\xi''_k)^2(v_j-v_{j+1})^{2H_2}\right]\right\}
\end{split}
\end{equation}
with $\gamma\geq 0$, $\kappa>0$ and $n=1,2,\ldots$, where $\xi'_1,\xi'_2,\ldots,\xi'_n$ and $\xi''_1,\xi''_2,\ldots,\xi''_n$ are two arbitrary rearrangements of the set $\{\xi_1,\xi_2,\ldots,\xi_n\}$.

\begin{lemma}\label{lem5.1}
Let $0<H_1,H_2<1$ and $\frac1{H_1}+\frac1{H_2}>3$. Then we have
\begin{equation}\label{sec5-eq5.400-01}
\widetilde{\Lambda}(t,t',n,\gamma)\leq C(t'-t)^{n\theta},\quad n=1,2,\ldots
\end{equation}
with $0<\theta\leq 1-H_2\left(\frac1q+2(1+\gamma)(1-\alpha)\right)$, provided $0< \alpha<1$, $q>1$ and
$$
\begin{cases}
2(1+\gamma)(1-\alpha)<\frac{1}{H_2}-\frac1q,&\\
2(1+\gamma)\alpha<\frac{1}{H_1}-1+\frac1q.
\end{cases}
$$
In particular, we can take $\theta=\frac1{H_1}\left(\alpha(H_1+H_2)-H_2\right)$ with $\frac{H_2}{H_1+H_2}<\alpha<1$ and $q=(1-\alpha)^{-1}$, provided  $2\gamma<(\frac{1}{H_1}+\frac{1}{H_2}-3)\wedge 2$.
\end{lemma}

The above lemma will be proved in Appendix~\ref{app2}.

\begin{proof}[Proof of Theorem~\ref{th5.1}]
By Kolmogorov continuity criterion, we first show that the estimates
\begin{equation}\label{sec5-eq5.3}
E\left|\ell'_{t',\varepsilon'}(a')-\ell'_{t,\varepsilon}(a)\right|^n\leq C|(t',\varepsilon',a')-(t,\varepsilon,a)|^{n\lambda},\quad n=2,4,\ldots
\end{equation}
hold for all $t,t'\in [0,T],a,a'\in {\mathbb R},\varepsilon,\varepsilon'>0$ and some $\lambda\in (0,1]$, where $|\cdot|$ denotes the Euclidean distance in ${\mathbb R}^3$. The estimate~\eqref{sec5-eq5.3} will be done in three parts and denote ${\mathbb D}_{l,l'}=\{l\leq r<s\leq l'\}$.

{\bf Step I.} We need to obtain the estimates
$$
E\left|\ell'_{t,\varepsilon'}(a)-\ell'_{t,\varepsilon}(a)\right|^n\leq C |\varepsilon'-\varepsilon|^{n\lambda},\quad n=2,4,\ldots
$$
for all $\varepsilon,\varepsilon'>0,a\in {\mathbb R}$ and some $\lambda\in [0,1]$. By~\eqref{sec3-eq3.3} we have
\begin{equation}\label{sec5-eq5.4}
\begin{split}
&E\left|\ell'_{t,\varepsilon'}(a)-\ell'_{t,\varepsilon}(a)\right|^n\\ &=\frac{1}{(2\pi)^n}\Bigl|\int_{({\mathbb D}_{0,t})^n}dr_1\cdots dr_nds_1\cdots ds_n\int_{\mathbb{R}^n} E\prod_{j=1}^n
\exp\left\{i\xi_j(B_{r_j}^{H_{1}}-B_{s_j}^{H_{2}})\right\}\\
&\qquad\cdot \prod_{j=1}^n\xi_j\left(\exp\left\{-
\frac{1}{2}\varepsilon'\xi^{2}_j\right\}-\exp\left\{ -\frac{1}{2}\varepsilon\xi^{2}_j\right\}\right)
e^{-i\xi_ja}d\xi_j\Bigr|\\
&\leq C |\varepsilon'-\varepsilon|^{n\lambda}
\int_{({\mathbb D}_{0,t})^n}dr_1\cdots dr_nds_1\cdots ds_n\int_{\mathbb{R}^n}
\prod_{j=1}^n|\xi_j|^{1+2\lambda}d\xi_j\\
&\qquad\cdot
\Bigl|E\prod_{j=1}^n\exp\left\{i\xi_j(B_{r_j}^{H_{1}}-B_{s_j}^{H_{2}})\right\}\Bigr|\\
&\equiv C |\varepsilon'-\varepsilon|^{n\lambda}\Lambda(0,t,n,2\lambda)
\end{split}
\end{equation}
for all $n=2,4,\ldots$ and some $\lambda\in (0,1)$ by the inequality
$$
\left|e^{-\varepsilon x}-e^{-\varepsilon' x}\right|\leq Cx^{\lambda}|\varepsilon-\varepsilon'|^{\lambda}
$$
for all $x>0$ and $\lambda\in (0,1)$.

Now, we need to prove $\Lambda(0,t,n,2\lambda)<\infty$. Let us first consider the expectation of product. This expectation in the integrand will take different forms over different regions of integration, depending on the ordering of $r_j$ and $s_j$. Fix such an ordering and let $u_1>u_2>\cdots>u_n$ and $v_1>v_2>\cdots>v_n$ be the relabeling of the sets $\{r_1,r_2,\ldots,r_n\}$ and $\{s_1,s_2,\ldots,s_n\}$ respectively. Thanks to  the local nondeterminism of fBm, we get
\begin{align*}
{\rm Var}\left(\sum_{j=1}^n\eta_jB_{u_j}^{H}\right)&={\rm Var}\left(\sum_{j=1}^{n-1}\sum_{k=1}^j\eta_k \left(B_{u_j}^{H}-B_{u_{j+1}}^{H}\right) +\sum_{k=1}^n\eta_kB^{H}_{u_n}\right)\\
&\geq \kappa\sum_{j=1}^{n-1}(\sum_{k=1}^j\eta_k )^2(u_j-u_{j+1})^{2H}+\kappa (\sum_{k=1}^n\eta_k)^2(u_n)^{2H}
\end{align*}
for a constant $\kappa>0$ and any fBm $B^H$ with Hurst index $H\in (0,1)$, where
$\eta_1,\eta_2,\ldots,\eta_n$ are some real numbers. It follows that

\begin{align*}
{\rm Var}\left(\sum_{j=1}^n\xi_j(B_{r_j}^{H_{1}}-\tilde{B}_{s_j}^{H_{2}}) \right)&=
{\rm Var}\left(\sum_{j=1}^n\xi_jB_{r_j}^{H_{1}}\right)+{\rm Var}\left(\sum_{j=1}^n\xi_j\tilde{B}_{s_j}^{H_{2}}\right)\\
&={\rm Var}\left(\sum_{j=1}^n\xi'_jB_{u_j}^{H_1}\right)+{\rm Var}\left(\sum_{j=1}^n\xi''_j\tilde{B}_{v_j}^{H_2}\right)\\
&\geq \kappa\sum_{j=1}^{n-1}(\sum_{k=1}^j\xi'_k )^2(u_j-u_{j+1})^{2H_1}+\kappa (\sum_{k=1}^n\xi'_k)^2(u_n)^{2H_1}\\
&\qquad+\kappa\sum_{j=1}^{n-1}(\sum_{k=1}^j\xi''_k )^2(v_j-v_{j+1})^{2H_2}+\kappa (\sum_{k=1}^n\xi''_k)^2(v_n)^{2H_2},
\end{align*}
where $\xi'_1,\xi'_2,\ldots,\xi'_n$ and $\xi''_1,\xi''_2,\ldots,\xi''_n$ are two rearrangements of the set $\{\xi_1,\xi_2,\ldots,\xi_n\}$ such that $\{(\xi'_1,u_1),(\xi'_2,u_2),\ldots,(\xi'_n,u_n)\}$ and $\{(\xi''_1,v_1),(\xi''_2,v_2),\ldots,(\xi''_n,v_n)\}$ are the rearrangements of $\{(\xi_1,r_1),(\xi_2,r_2),\ldots,(\xi_n,r_n)\}$ and $\{(\xi_1,s_1),(\xi_2,s_2),\ldots,(\xi_n,s_n)\}$ via the second coordinates, respectively. Combining this with Lemma~\ref{lem5.1}, we get
\begin{align*}
\Lambda(0,t,n,2\lambda)
&\leq C \int_{\widetilde{{\mathbb D}}_{0,t}(u)}du_1\cdots du_n\int_{\widetilde{{\mathbb D}}_{0,t}(v)} dv_1\cdots dv_n \int_{\mathbb{R}^n}\prod_{j=1}^n|\xi_j|^{1+2\lambda}d\xi_j\\
&\qquad\cdot
\exp\left\{-\frac12\kappa\left[(\sum\limits_{k=1}^n\xi'_k )^2(u_n)^{2H_1}
+(\sum\limits_{k=1}^n\xi''_k )^2(v_n)^{2H_2}\right]\right\}\\
&\qquad\cdot
\prod_{j=1}^{n-1}\exp\left\{-\frac12 \kappa\left[(\sum\limits_{k=1}^{j}\xi'_k)^2(u_j-u_{j+1})^{2H_1} +(\sum\limits_{k=1}^{j}\xi''_k)^2(v_j-v_{j+1})^{2H_2}\right]\right\}\\
&=C \widetilde{\Lambda}(0,t,n,2\lambda)<\infty
\end{align*}
for $t\in [0,T]$ and $n=2,4,\ldots$, provided $0<4\lambda<\frac{1}{H_1}+\frac{1}{H_2}-3$, which proves that the estimate
$$
E\left|\ell'_{t,\varepsilon'}(a)-\ell'_{t,\varepsilon}(a)\right|^n\leq C |\varepsilon'-\varepsilon|^{n\lambda}
$$
holds for $n=2,4,\ldots$ and $t\in [0,T]$, if we choose $\lambda$ so that $0<4\lambda<\frac{1}{H_1}+\frac{1}{H_2}-3$.

{\bf Step II.} We obtain the estimate
\begin{equation}\label{sec5-eq5.5}
E\left|\ell'_{t,\varepsilon}(b)-\ell'_{t,\varepsilon}(a)\right|^n\leq C |b-a|^{n\lambda},\quad n=2,4,\ldots
\end{equation}
for all $\varepsilon>0$, $a,b\in {\mathbb R}$, $t\geq 0$ and some $\lambda\in (0,1)$. We have, for all $a,b\in {\mathbb R}$ and $t\in [0,T]$
\begin{equation}\label{sec5-eq5.7}
\begin{split}
E|\ell'_{t,\varepsilon}(b)-&\ell'_{t,\varepsilon}(a)|^n\\ &=\frac{1}{(2\pi)^n}\Bigl|\int_{({\mathbb D}_{0,t})^n}dr_1\cdots dr_nds_1\cdots ds_n\int_{\mathbb{R}^n}E\prod_{j=1}^n
\exp\left\{i\xi_j(B_{r_j}^{H_{1}}-\tilde{B}_{s_j}^{H_{2}})\right\}\\
&\qquad\cdot \prod_{j=1}^n\xi_j\left(\exp\left\{-
ib\xi_j\right\}-\exp\left\{ -ia\xi_j\right\}\right)
e^{-\frac{1}{2}\varepsilon\xi^{2}_j}d\xi_j\Bigr|\\
&\leq C|b-a|^{n\lambda}\int_{({\mathbb D}_{0,t})^n}dr_1\cdots dr_nds_1\cdots ds_n\int_{\mathbb{R}^n}\prod_{j=1}^n|\xi_j|^{1+\lambda}d\xi_j\\
&\qquad\cdot\Bigl|E\prod_{j=1}^n
\exp\left\{i\xi_j(B_{r_j}^{H_{1}}-\tilde{B}_{s_j}^{H_{2}})\right\}\Bigr|\\
&= C|b-a|^{n\lambda} \Lambda(0,t,n,\lambda)
\end{split}
\end{equation}
for all $n=1,2,\ldots$ and $\lambda\in [0,1]$ by the inequality
$$
\left|e^{-ixb}-e^{-ixa}\right|\leq C|x|^{\lambda}|b-a|^{\lambda}
$$
for all $x\in {\mathbb R}$ and $\lambda\in [0,1]$. This shows that
the estimate~\eqref{sec5-eq5.5} holds for $n=2,4,\ldots$ and $t\in [0,T]$, if we choose $\lambda$ so that $2\lambda<\frac{1}{H_1}+\frac{1}{H_2}-3$.

{\bf Step III.} We obtain the estimate
\begin{equation}\label{sec5-eq5.6}
E\left|\ell'_{t',\varepsilon}(a)-\ell'_{t,\varepsilon}(a)\right|^n\leq C |t'-t|^{n\lambda},\quad n=2,4,\ldots
\end{equation}
for all $\varepsilon>0$, $a\in {\mathbb R}$, $t'>t\geq 0$ and some $\lambda\in (0,1)$. In order to prove the estimate~\eqref{sec5-eq5.6} we have, for all $a\in {\mathbb R},t,t'\in [0,T]$ and $t<t'$
\begin{equation}\label{sec5-eq5.8}
\begin{split}
E|\ell'_{t',\varepsilon}(a)-\ell'_{t,\varepsilon}(a)|^n &=\frac{1}{(2\pi)^n}\Bigl|\int_{({\mathbb D}_{0,t'}\setminus{\mathbb D}_{0,t})^n}dr_1\cdots dr_nds_1\cdots ds_n\int_{\mathbb{R}^n}
\prod_{j=1}^n\xi_je^{-i\xi_j a}e^{-\frac1{2}\varepsilon{\xi^{2}_j}}d\xi_j\\
&\qquad\cdot E\prod_{j=1}^n
\exp\left\{i\xi_j(B_{r_j}^{H_{1}}-\tilde{B}_{s_j}^{H_{2}})\right\}\Bigr|\\
&\leq \int_{[t,t']^n}\int_{[0,s_1]\times [0,s_2]\times\cdots\times [0,s_n]}dr_1\cdots dr_nds_1\cdots ds_n\int_{\mathbb{R}^n}\prod_{j=1}^n|\xi_j|d\xi_j\\
&\qquad\cdot \Bigl|E\prod_{j=1}^n
\exp\left\{i\xi_j(B_{r_j}^{H_{1}}-\tilde{B}_{s_j}^{H_{2}})\right\}\Bigr|\\
\end{split}
\end{equation}
for all $n=2,4,\ldots$. It follows from Step I and Lemma~\ref{lem5.1} that
\begin{align*}
E\left|\ell'_{t'}(a)-\ell'_{t}(a)\right|^n\leq C \widetilde{\Lambda}(t,t',n,0) \leq C(t'-t)^{n\lambda}
\end{align*}
for all $0<\lambda\leq \frac1{H_1}\left(\alpha(H_1+H_2)-H_2\right)$ with $\frac{H_2}{H_1+H_2}<\alpha<1$. Thus, we have obtained the desired estimate~\eqref{sec5-eq5.3} and the limit
$$
\ell'_{t}(a)=\lim_{\varepsilon\to 0}\ell'_{t,\varepsilon}(a)
$$
exists almost surely, and in $L^p(\Omega)$ for all $p\in (0,\infty)$.

Finally, as two direct consequences of Step II and Step III we see that
\begin{equation}\label{sec5-eq5.50}
E\left|\ell'_{t}(b)-\ell'_{t}(a)\right|^n\leq C |b-a|^{n\lambda},\quad n=2,4,\ldots
\end{equation}
and
\begin{equation}\label{sec5-eq5.60}
E\left|\ell'_{t'}(a)-\ell'_{t}(a)\right|^n\leq C |t'-t|^{n\beta},\quad n=2,4,\ldots
\end{equation}
for all $a,b\in {\mathbb R}$ and $t',t\geq 0$ if we choose $\lambda$ and $\beta$ so that $2\lambda<\frac{1}{H_1}+\frac{1}{H_2}-3$ and $0<\beta\leq \frac1{H_1}\left(\alpha(H_1+H_2)-H_2\right)$ with $\frac{H_2}{H_1+H_2}<\alpha<1$. These show that $\ell'_{t}(a)$ exists in $L^p(\Omega)$ for all $p>0$, $t\in [0,T]$, $a\in {\mathbb R}$, and has a modification which is a.s. jointly H\"older continuous in $(a,t)$.
\end{proof}

\begin{theorem}\label{th5.2}
Let $0<H_1,H_2<1$. The processes $\ell_{\varepsilon,t}(a),\varepsilon>0$ converges almost surely, and in $L^p(\Omega)$ for all $p\in (0,\infty)$, as $\varepsilon$ tends to zero. Moreover, the process $\ell_t(a)$ has a modification which is a.s. jointly H\"older continuous in $(a,t)$. In particular, $a\mapsto \ell_t(a)$ is H\"older continuous of order $\gamma<\min\{\frac{1}{2H_1}+\frac{1}{2H_2}-\frac12,1\}$.
\end{theorem}

\begin{proof}
In a same way as proof of Theorem~\ref{th5.1}, one can obtain the estimate
\begin{equation}\label{sec5-eq5.3-1}
E\left|\ell_{t',\varepsilon'}(a')-\ell_{t,\varepsilon}(a)\right|^n\leq C|(t',\varepsilon',a')-(t,\varepsilon,a)|^{n\beta},\quad n=2,4,\ldots
\end{equation}
for all $t,t'\in [0,T],a,a'\in {\mathbb R},\varepsilon,\varepsilon'>0$ and some $\beta\in (0,1]$. In particular, we have
\begin{align*}
E|\ell_{t,\varepsilon}(a')-\ell_{t,\varepsilon}(a)|^n &=\frac{1}{(2\pi)^n}\Bigl|\int_{({\mathbb D}_{0,t})^n}dr_1\cdots dr_nds_1\cdots ds_n\int_{\mathbb{R}^n}E\prod_{j=1}^n
\exp\left\{i\xi_j(B_{r_j}^{H_{1}}-\tilde{B}_{s_j}^{H_{2}})\right\}\\
&\qquad\cdot \prod_{j=1}^n \left(\exp\left\{-
ia'\xi_j\right\}-\exp\left\{ -ia\xi_j\right\}\right)
e^{-\frac{1}{2}\varepsilon\xi^{2}_j}d\xi_j\Bigr|\\
&\leq C|a'-a|^{n\lambda}\int_{({\mathbb D}_{0,t})^n}dr_1\cdots dr_nds_1\cdots ds_n\int_{\mathbb{R}^n}\prod_{j=1}^n|\xi_j|^{\lambda}d\xi_j\\
&\qquad\cdot\Bigl|E\prod_{j=1}^n
\exp\left\{i\xi_j(B_{r_j}^{H_{1}}-\tilde{B}_{s_j}^{H_{2}})\right\}\Bigr|\\
&\leq C|a'-a|^{n\lambda}
\end{align*}
for all $0<2\lambda<\left(\frac{1}{H_1}+\frac{1}{H_2}-1\right)\wedge 2$, and the theorem follows.
\end{proof}
\begin{corollary}
Let $0<H_1,H_2<1$. If $\frac1{H_1}+\frac1{H_2}>3$, we have $\ell'_t(a)=\frac{\partial}{\partial a}\ell_t(a)$ a.s. for all $t\geq 0$ and $a\in {\mathbb R}$, i.e. $\ell_t(a)$ is differentiable in $a$ for all $t\geq 0$ and
$$
\frac{\partial}{\partial a}\ell_t(a)=\lim_{\varepsilon\to 0}\ell'_{t,\varepsilon}(a),
$$
almost surely, and in $L^p(\Omega)$ with $p>0$.
\end{corollary}
\begin{proof}
It is clear that $\ell'_{t,\varepsilon}(a)=\frac{\partial}{\partial a}\ell_{t,\varepsilon}(a)$ for any $\varepsilon,t>0,a\in {\mathbb R}$ and hence
$$
\ell_{t,\varepsilon}(a)=\ell_{t,\varepsilon}(b) +\int_b^a\ell'_{t,\varepsilon}(x)dx
$$
for all $a,b\in {\mathbb R}$ and $\varepsilon>0$. On the other hand, the estimates~\eqref{sec5-eq5.3} and~\eqref{sec5-eq5.3-1} assure that a locally uniform and hence continuous limits
\begin{equation}\label{sec5-eq5.3-0}
\ell'_{t}(a)=\lim_{\varepsilon\to 0}\ell'_{t,\varepsilon}(a)
\end{equation}
\begin{equation}\label{sec5-eq5.3-00}
\ell_{t}(a)=\lim_{\varepsilon\to 0}\ell_{t,\varepsilon}(a)
\end{equation}
hold. The locally uniform convergence~\eqref{sec5-eq5.3-0} and~\eqref{sec5-eq5.3-00} imply that
$$
\ell_{t}(a)=\ell_{t}(b) +\int_b^a\ell'_{t}(x)dx
$$
and therefore $\frac{\partial}{\partial a}\ell_{t}(a)=\ell'_{t}(a)$. This completes the proof.
\end{proof}
\begin{theorem}\label{th5.3}
Let $0<H_1,H_2<1$.

$(1)$ If $\frac1{H_1}+\frac1{H_2}>3$, we then have \begin{equation}\label{sec5-eq5.12}
\int_0^t ds\int_0^sf'(B^{H_1}_r-\tilde{B}^{H_2}_s)dr=-\int_{\mathbb R}f(a)\ell'_t(a)da
\end{equation}
for any $f\in C^1({\mathbb R})$ and $t\in [0,T]$.

$(2)$ If $f$ is continuous, then
\begin{equation}\label{sec5-eq5.222}
\int_0^t \int_0^sf(B^{H_1}_r-\tilde{B}^{H_2}_s)drds=\int_{\mathbb R}f(a)\ell_t(a)da
\end{equation}
for all $t\in [0,T]$.
\end{theorem}
\begin{proof}[Proof of Theorem~\ref{th5.3}]
By the locally $L^1$ convergence in~\eqref{sec5-eq5.3-0} and noting that both $B^{H_1}$ and $\tilde{B}^{H_2}$ are a.s. bounded on $[0,t]$, we see that the following manipulations hold:
\begin{align*}
\int_{\mathbb R}f(a)\ell'_t(a)da&=\lim_{\varepsilon\to 0}\int_{\mathbb R}f(a)\ell'_{t,\varepsilon}(a)da\\
&=-\lim_{\varepsilon\to 0}\int_{\mathbb R}f(a)da\int_0^t ds \int_0^sp'_\varepsilon(B^{H_1}_r-\tilde{B}^{H_2}_s-a)dr\\
&=-\lim_{\varepsilon\to 0}\int_0^t ds \int_0^sdr\int_{\mathbb R}f'(a)p_\varepsilon(B^{H_1}_r-\tilde{B}^{H_2}_s-a)da\\
&=-\lim_{\varepsilon\to 0}\int_0^t ds \int_0^sdr(f'\ast p_\varepsilon)(B^{H_1}_r-\tilde{B}^{H_2}_s)\\
&=-\int_0^t ds \int_0^sdrf'(B^{H_1}_r-\tilde{B}^{H_2}_s)
\end{align*}
for any $f\in C^1({\mathbb R})$ with compact support. Since both $B^{H_1}$ and $\tilde{B}^{H_2}$ are bounded a.s. we have that $a\mapsto \ell'_t(a)$ has compact support a.s. so that the above manipulations hold for all $C^1$-functions $f$.

Similarly, one can obtain the identity~\eqref{sec5-eq5.222}, and the theorem follows.
\end{proof}

\section{The hybrid quadratic covariation, case $H_1<\frac12$}\label{sec6}
In this section we throughout let $H_1\leq H_2$, and inspired by the occupation formula~\eqref{sec5-eq5.12}, our main aim is to obtain the following Bouleau-Yor type identity
\begin{equation}\label{sce7-eq7-0011}
[f(B^{H_1}-\tilde{B}^{H_2}),B^{H_1}]^{(HC)}_t=-\int_{\mathbb R}f(a)\ell'_t(a)da,\quad t\geq 0
\end{equation}
for some suitable Borel functions $f$. Recall that the
quadratic covariation $[f(B),B]$ of Brownian motion $B$ can be
characterized as
\begin{equation}\label{sec8-eq8.1}
[f(B),B]_t=-\int_{\mathbb R}f(a){\mathscr L}^{B}(da,t),
\end{equation}
where $f$ is locally square integrable and ${\mathscr L}^{B}(x,t)$
is the local time of Brownian motion. This is called the Bouleau-Yor
identity. More works for this can be found in Bouleau-Yor~\cite{Bouleau}, Eisenbaum~\cite{Eisen1,Eisen2}, Errami-Russo~\cite{Errami-Russo}, Feng--Zhao~\cite{Feng1,Feng3}, Follmer {\em et al}~\cite{Follmer}, Gradinaru {\em et al}~\cite{Grad2}, Moret-Nualart~\cite{Moret}, Rogers--Walsh~\cite{Rogers2}, Russo-Vallois~\cite{Russo-Vallois1,Russo-Vallois2}, Yan {\it et al}~\cite{Yan4,Yan3} and the references therein.

Recall that
$$
J_\varepsilon(H_1,H_2,t,f)=\frac1{\varepsilon^{2H_1}} \int_0^t ds\int_0^s \left\{f(B^{H_1}_{r+\varepsilon}-\tilde{B}^{H_2}_s) -f(B^{H_1}_{r}-\tilde{B}^{H_2}_s)\right\} \left(B^{H_1}_{r+\varepsilon}-B^{H_1}_r\right)dr
$$
for any Borel function $f$ and the HQC $[f(B^{H_1}-\tilde{B}^{H_2}),B^{H_1}]^{(HC)}_t$ is defined by
\begin{equation}\label{sec7-eq7.3}
[f(B^{H_1}-\tilde{B}^{H_2}),B^{H_1}]^{(HC)}_t=\lim_{\varepsilon\downarrow 0}J_\varepsilon(H_1,H_2,t,f),
\end{equation}
provided the limit exists in $L^1(\Omega)$.

\begin{corollary}\label{cor9.1}
If $f\in C^1({\mathbb R})$, we then have
\begin{equation}\label{sec7-eq7.4}
[f(B^{H_1}-\tilde{B}^{H_2}),B^{H_1}]^{(HC)}_t =\int_0^t ds\int_0^sf'(B^{H_1}_r-\tilde{B}^{H_2}_s)dr
\end{equation}
for all $t\in [0,T]$. In particular, we have
$$
[B^{H_1}-\tilde{B}^{H_2},B^{H_1}]^{(HC)}_t=\frac{1}{2}t^2.
$$
\end{corollary}
As a direct consequence of the above corollary and~\eqref{sec5-eq5.12}, we have
\begin{equation}\label{sec7-eq7.5}
[f(B^{H_1}-\tilde{B}^{H_2}),B^{H_1}]^{(HC)}_t=-\int_{\mathbb R}f(a)\ell'_t(a)da
\end{equation}
for all $f\in C^1({\mathbb R})$ and $t\geq 0$. In order to prove the existence of the HQC, we decompose $J_\varepsilon(H_1,H_2,t,f)$ as follows:
\begin{align*}
J_\varepsilon(H_1,H_2,t,f)&=\frac1{\varepsilon^{2H_1}} \int_0^t ds\int_0^s f(B^{H_1}_{r+\varepsilon}-\tilde{B}^{H_2}_s) \left(B^{H_1}_{r+\varepsilon}-B^{H_1}_r\right)dr\\
&\qquad\qquad\qquad-\frac1{\varepsilon^{2H_1}}\int_0^t ds\int_0^s f(B^{H_1}_{r}-\tilde{B}^{H_2}_s)\left(B^{H_1}_{r+\varepsilon}-B^{H_1}_r\right)dr\\
&\equiv J^{+}_\varepsilon(H_1,H_2,t,f)-J^{-}_\varepsilon(H_1,H_2,t,f)
\end{align*}
for all $\varepsilon>0$ and $t\in [0,T]$, and consider the set
$$
{\mathscr H}=\{f\,:\,{\text { measurable
functions on ${\mathbb R}$ such that $\|f\|_{\mathscr
H}<\infty$}}\},
$$
where
\begin{align*}
\|f\|_{\mathscr H}^2:=\int_0^T \int_0^s\int_{\mathbb
R}|f(x)|^2e^{-\frac{x^2}{2(r^{2H_1}+s^{2H_2})}}\frac{dxdrds}{ \sqrt{2\pi(r^{2H_1}+s^{2H_2})}}.
\end{align*}
Then, ${\mathscr H}$ is a Banach space. In fact, we have
${\mathscr H}=L^2({\mathbb R},\mu(dx))$ with
$$
\mu(dx)=\left(\int_0^T\int_0^se^{-\frac{x^2}{2(r^{2H_1}+s^{2H_2})}} \frac{dr ds}{\sqrt{2\pi(r^{2H_1}+s^{2H_2})}}\right)dx
$$
and $\mu({\mathbb R})=\frac12T^2<\infty$, which implies that the set ${\mathscr E}$ of elementary functions of the form
$$
f_\triangle(x)=\sum_{i}f_{i}1_{(x_{i-1},x_{i}]}(x)
$$
is dense in ${\mathscr H}$, where $\{x_i,0\leq i\leq l\}$ is an
finite sequence of real numbers such that $x_i<x_{i+1}$.

Recall that the space ${\mathcal C }^\nu$ of $\nu$-H\"older continuous functions $f : [0, T] \to {\mathbb R}$, equipped with the norm
$$
\|f\|_{(\nu)}:=\|f\|_{\infty}+\sup_{x,y\in {\mathbb R}}\frac{|f(x)-f(y)|}{|x-y|^\nu}<\infty,
$$
where $\mu\in (0,1]$ and $\|f\|_{\infty}=\sup_{x\in \mathbb{R}}|f(x)|$.
\begin{corollary}
For all $H_1,H_2\in (0,1)$, we have ${\mathscr H}\supset {\mathcal C }^\nu$ for all $\nu\in (0,1]$.
\end{corollary}

For simplicity we let $T=1$ in the following discussions.
\begin{theorem}\label{th7.1}
Let $0<H_1<\frac12$ and $f\in {\mathscr H}$. Then, the HQC $[f(B^{H_1}-\tilde{B}^{H_2}),B^{H_1}]^{(HC)}$ exists in $L^2(\Omega)$ and
\begin{align}\label{sec7-eq7.6}
E\left|[f(B^{H_1}-\tilde{B}^{H_2}),B^{H_1}]^{(HC)}_t\right|^2\leq C_{H_1,H_2,T}\|f\|_{\mathscr H}^2
\end{align}
for all $t\in [0,1]$.
\end{theorem}

In order to prove the theorem we claim that the following two statements:
\begin{itemize}
\item for any $\varepsilon>0$, $t\in [0,1]$, and $f\in {\mathscr H}$, $J_\varepsilon^{\pm}(H_1,H_2,f,t)\in L^2(\Omega)$. That is,
\begin{align}\label{sec7-eq7-101}
&E\left|J_\varepsilon^{-}(H_1,H_2,f,t)\right|^2\leq C_{H_1,H_2,T}\|f\|_{\mathscr H}^2,\\  \label{sec7-eq7-102}
&E\left|J_\varepsilon^{+}(H_1,H_2,f,t)\right|^2\leq C_{H_1,H_2}\|f\|_{\mathscr H}^2.
\end{align}
\item $J_{\varepsilon}^{-}(H_1,H_2,f,t)$ and $J_{\varepsilon}^{+}(H_1,H_2,f,t)$ are two Cauchy's sequences in $L^2(\Omega)$ for all $t\in [0,1]$. That is,
\begin{equation}\label{sec40-eq3-1}
E\left|J_{\varepsilon_1}^{-}(H_1,H_2,f,t)-J_{\varepsilon_2}^{-} (H_1,H_2,f,t)\right|^2
\longrightarrow 0,
\end{equation}
and
\begin{equation}\label{sec40-eq3-2}
E\left|J_{\varepsilon_1}^{+}(H_1,H_2,f,t)-J_{\varepsilon_2}^{+} (H_1,H_2,f,t)\right|^2
\longrightarrow 0
\end{equation}
for all $t\in [0,1]$, as $\varepsilon_1,\varepsilon_2\downarrow 0$.
\end{itemize}

We split the proof of two statements into two parts and this is similar to Yan {\em et al.}~\cite{Yan7}. Let $T=1$ for simplicity and we also need the next lemmas which are some elementary calculations.

\begin{lemma}\label{lemA1-4}
Let $\lambda_{r,s},\mu$ and $\rho^2$ be defined in Section~\ref{sec2}. If $f\in C^\infty({\mathbb R})$ admit compact support. Then we have
\begin{equation}\label{app1-eq1.1}
\begin{split}
|E&\left[f'(B^{H_1}_r-\tilde{B}^{H_2}_s) f'(B^{H_1}_{r'}-\tilde{B}^{H_2}_{s'})\right]|\\
&\qquad\leq
\frac{\sqrt{\lambda_{r',s'}\lambda_{r,s}}}{\rho^2} \left(E\left[|f(B^{H_1}_r-\tilde{B}^{H_2}_s)|^2 \right]E\left[|f(B^{H_1}_{r'}-\tilde{B}^{H_2}_{s'})|^2\right]\right)^{1/2}
\end{split}
\end{equation}
and
\begin{equation}\label{app1-eq1.2}
\begin{split}
|E&\left[f''(B^{H_1}_r-\tilde{B}^{H_2}_s) f(B^{H_1}_{r'}-\tilde{B}^{H_2}_{s'})\right]|\\
&\qquad\leq \frac{3\lambda_{r',s'}}{\rho^2} \left(E\left[|f(B^{H_1}_r-\tilde{B}^{H_2}_s)|^2 \right]
E\left[|f(B^{H_1}_{r'}-\tilde{B}^{H_2}_{s'})|^2\right]\right)^{1/2}
\end{split}
\end{equation}
for all $(r,s,r',s')\in {\mathbb T}=\{0<r<s<t,\;0<r'<s'<t\}$.
\end{lemma}

\begin{lemma}[Yan {\em et al.}~\cite{Yan7}]\label{lemA1-3}
Let $0<H<\frac12$.
\begin{itemize}
\item [(1)] For all $t>s>r>0$ we have
\begin{equation}\label{sec2-Lemma1-1}
\begin{split}
&\left|E\left[B^H_t(B^H_{t}-B^H_{s})\right]\right|\leq (t-s)^{2H},\\
&\left|E\left[B^H_t(B^H_{s}-B^H_{r})\right]\right|\leq (s-r)^{2H},\\
&\left|E\left[B^H_r(B^H_{t}-B^H_{s})\right]\right|\leq (t-s)^{2H},\\
&\left|E\left[B^H_s(B^H_{t}-B^H_{r})\right]\right|\leq (t-s)^{2H}.
\end{split}
\end{equation}

\item [(2)] For all $0<s'<t'<s<t$ we have
\begin{equation}\label{app1-eq1.3}
\left|E\left[(B^H_t-B^H_s)(B^H_{t'}-B^H_{s'})\right]\right|\leq
\frac{(t-s)^{2H}(t'-s')^{2H}}{(s-t')^{2H}}.
\end{equation}
\end{itemize}
\end{lemma}

\begin{proof}[Part I : Proof of the estimates~\eqref{sec7-eq7-101} and~\eqref{sec7-eq7-102}]

We have
\begin{align*}
E|J_\varepsilon^{-}&(H_1,H_2,f,t)|^2=\frac{1}{\varepsilon^{4H_1}}
\int_0^t\int_0^t dsds'\\
&\cdot\int_0^s\int_0^{s'}  E\left[f(B^{H_1}_{r}-\tilde{B}^{H_2}_s) f(B^{H_1}_{r'}-\tilde{B}^{H_2}_{s'}) (B^{H_1}_{r+\varepsilon}-B^{H_1}_r) (B^{H_1}_{r'+\varepsilon}-B^{H_1}_{r'})\right]drdr'
\end{align*}
for all $\varepsilon>0$ and $t\geq 0$. Now, let us estimate the express
$$
\Phi_{\varepsilon_1,\varepsilon_2}(r,s,r',s';H_1,H_2):
=E\left[f(B^{H_1}_r-\tilde{B}^{H_2}_s)f(B^{H_1}_{r'}-\tilde{B}^{H_2}_{s'}) (B^{H_1}_{r+\varepsilon_1}-B^{H_1}_r)(B^{H_1}_{r'+\varepsilon_2} -B^{H_1}_{r'})\right]
$$
for all $\varepsilon_1,\varepsilon_2>0$ and $s>r>0,s'>r'>0$. Thus, it is enough to obtain the estimates~\eqref{sec7-eq7-101} and~\eqref{sec7-eq7-102} for all $f\in {\mathscr E}$. By approximating we can assume that $f$ is an infinitely differentiable function with compact support. It follows from the duality relationship~\eqref{sec2-eq2.1} that
\begin{equation}\label{sec7-eq7.9}
\begin{split}
&\Phi_{\varepsilon_1,\varepsilon_2}(r,s,r',s';H_1,H_2)\\ &=E\left[f(B^{H_1}_r-\tilde{B}^{H_2}_s)f(B^{H_1}_{r'}-\tilde{B}^{H_2}_{s'}) (B^{H_1}_{r+\varepsilon_1}-B^{H_1}_r)\int_{r'}^{r'+\varepsilon_2}dB^{H_1}_l \right]\\
&=E\left[B^{H_1}_r(B^{H_1}_{r'+\varepsilon_2}-B^{H_1}_{r'})\right]E\left[
f'(B^{H_1}_r-\tilde{B}^{H_2}_s)f(B^{H_1}_{r'}-\tilde{B}^{H_2}_{s'}) (B^{H_1}_{r+\varepsilon_1}-B^{H_1}_r)\right]\\
&\;\; +E\left[B^{H_1}_{r'}(B^{H_1}_{r'+\varepsilon_2}-B^{H_1}_{r'})\right]E\left[
f(B^{H_1}_r-\tilde{B}^{H_2}_s)f'(B^{H_1}_{r'}-\tilde{B}^{H_2}_{s'}) (B^{H_1}_{r+\varepsilon_1}-B^{H_1}_r)\right]\\
&\;\;
+E\left[(B^{H_1}_{r+\varepsilon_1}-B^{H_1}_r)(B^{H_1}_{r'+\varepsilon_2} -B^{H_1}_{r'})\right] E\left[f(B^{H_1}_r-\tilde{B}^{H_2}_s)f(B^{H_1}_{r'}-\tilde{B}^{H_2}_{s'}) \right]\\
&=E\left[B^{H_1}_r(B^{H_1}_{r'+\varepsilon_2}-B^{H_1}_{r'})\right]
E\left[B^{H_1}_r(B^{H_1}_{r+\varepsilon_1}-B^{H_1}_{r})\right]
E\left[
f''(B^{H_1}_r-\tilde{B}^{H_2}_s)f(B^{H_1}_{r'}-\tilde{B}^{H_2}_{s'})\right]\\
&\;\;+E\left[B^{H_1}_r(B^{H_1}_{r'+\varepsilon_2}-B^{H_1}_{r'})\right]
E\left[B^{H_1}_{r'}(B^{H_1}_{r+\varepsilon_1}-B^{H_1}_{r})\right]
E\left[
f'(B^{H_1}_r-\tilde{B}^{H_2}_s)f'(B^{H_1}_{r'}-\tilde{B}^{H_2}_{s'})\right]\\
&\;\;+E\left[B^{H_1}_{r'}(B^{H_1}_{r'+\varepsilon_2}-B^{H_1}_{r'})\right] E\left[B^{H_1}_{r}(B^{H_1}_{r+\varepsilon_1}-B^{H_1}_{r})\right]
E\left[
f'(B^{H_1}_r-\tilde{B}^{H_2}_s)f'(B^{H_1}_{r'}-\tilde{B}^{H_2}_{s'})\right]\\
&\;\;+E\left[B^{H_1}_{r'}(B^{H_1}_{r'+\varepsilon_2}-B^{H_1}_{r'})\right] E\left[B^{H_1}_{r'}(B^{H_1}_{r+\varepsilon_1}-B^{H_1}_{r})\right]
E\left[f(B^{H_1}_r-\tilde{B}^{H_2}_s)f''(B^{H_1}_{r'}-\tilde{B}^{H_2}_{s'})\right]\\
&\;\;+E\left[(B^{H_1}_{r+\varepsilon_1}-B^{H_1}_r) (B^{H_1}_{r'+\varepsilon_2} -B^{H_1}_{r'})\right] E\left[f(B^{H_1}_r-\tilde{B}^{H_2}_s)f(B^{H_1}_{r'}-\tilde{B}^{H_2}_{s'}) \right]\\
&\equiv \sum_{j=1}^5\Psi_j(r,s,r',s',\varepsilon_1,\varepsilon_2)
\end{split}
\end{equation}
for all $s>r>0,s'>r'>0$ and $\varepsilon_1,\varepsilon_2>0$. In order to end the proof we claim to estimate
\begin{equation}\label{sec7-eq7.10}
\Lambda_j:=\frac{1}{\varepsilon^{4H_1}}\left|\int_{({\mathbb D}_{0,t})^2} \Psi_j(r,s,r',s',\varepsilon,\varepsilon)dr'ds'drds\right|,\quad j=1,2,3,4,5
\end{equation}
for all $\varepsilon>0$ small enough, where ${\mathbb D}_{0,t}=\{0\leq r<s\leq t\}$

For $j=5$, from the fact
\begin{align*}
|E&\left[(B^{H_1}_{r+\varepsilon}-B^{H_1}_r)(B^{H_1}_{r'+\varepsilon} -B^{H_1}_{r'})\right]|\leq \varepsilon^{2H_1}< \frac{\varepsilon^{4H_1}}{(r-r')^{2H_1}}
\end{align*}
for $0<r-r'<\varepsilon$ and the estimate~\eqref{app1-eq1.3} we have
\begin{align*}
\Lambda_5&\leq \frac{1}{\varepsilon^{4H_1}}\int_0^t\int_0^t dsds'\int_0^sdr\int_0^{s'}dr'\\
&\qquad\cdot \left(E[|f(B^{H_1}_r-\tilde{B}^{H_2}_s)|^2]+
E[|f(B^{H_1}_{r'}-\tilde{B}^{H_2}_{s'})|^2]\right)
\left|E[(B^{H_1}_{r+\varepsilon}-B^{H_1}_r)(B^{H_1}_{r'+\varepsilon} -B^{H_1}_{r'})]\right|\\
&\leq C \int_0^t ds\int_0^s
E\left[\left|f(B^{H_1}_r-\tilde{B}^{H_2}_s)\right|^2\right]dr\int_0^t ds'\int_0^{s'}\frac{dr'}{|r-r'|^{2H_1}}\\
&\leq C \|f\|_{\mathscr H}^2
\end{align*}
for all $0<\varepsilon\leq 1$ and $t\in [0,1]$.

For $j=1$, from Lemma~\ref{lemA1-3}, Lemma~\ref{lem2.5} and Lemma~\ref{lemA1-4} we have
\begin{align*}
\Lambda_1&\leq \int_{({\mathbb D}_{0,t})^2}
|E\left[
f''(B^{H_1}_r-\tilde{B}^{H_2}_s)f(B^{H_1}_{r'}-\tilde{B}^{H_2}_{s'})\right] |dr'ds'drds\\
&\leq \int_{({\mathbb D}_{0,t})^2} dr'ds'drds
\frac{3\lambda_{r',s'}}{\rho^2}
E\left[|f(B^{H_1}_{r}-\tilde{B}^{H_2}_{s})|^2\right]\\
&\qquad+\int_{({\mathbb D}_{0,t})^2} dr'ds'drds
\frac{3\lambda_{r',s'}}{\rho^2} E\left[|f(B^{H_1}_{r'}-\tilde{B}^{H_2}_{s'})|^2\right]\\
&\leq C \int_0^t ds\int_0^sdr E\left[|f(B^{H_1}_{r}-\tilde{B}^{H_2}_{s})|^2\right]
\int_0^t ds'\int_0^{s'}\frac{dr'}{|r-r'|^{2H_1}} \leq C \|f\|_{\mathscr H}^2
\end{align*}
for all $0<\varepsilon\leq 1$ and $t\in [0,1]$. Similarly, we can obtain the estimate~\eqref{sec7-eq7.10} for $j=2,3,4$, and the estimates~\eqref{sec7-eq7-101} follows.

Similarly one can prove the estimate~\eqref{sec7-eq7-102} and the first statement follows.
\end{proof}

\begin{proof}[Part II : Proof of the estimates~\eqref{sec40-eq3-1} and~\eqref{sec40-eq3-2}.]

Without loss of generality one may assume that $\varepsilon_1>\varepsilon_2$ and for $f\in {\mathscr H}$ we take the sequence $\{f_{\triangle,n}\}\subset {\mathscr E}$ such that $f_{\triangle,n}\to f$ in ${\mathscr H}$. Then we have
\begin{align*}
E|J_{\varepsilon_1}&(H_1,H_2,t,f)-J_{\varepsilon_2}(H_1,H_2,t,f)|^2\\
&\leq
3E|J_{\varepsilon_1}(H_1,H_2,t,f-f_{\triangle,n})|^2 +3E|J_{\varepsilon_2}(H_1,H_2,t,f-f_{\triangle,n})|^2\\
&\hspace{3cm}+3E|J_{\varepsilon_1}(H_1,H_2,t,f_{\triangle,n}) -J_{\varepsilon_2}(H_1,H_2,t,f_{\triangle,n})|^2\\
&\leq C_{H_1,H_2}\|f-f_{\triangle,n}\|_{\mathscr
H}^2+3E|J_{\varepsilon_1}(H_1,H_2,t,f_{\triangle,n}) -J_{\varepsilon_2}(H_1,H_2,t,f_{\triangle,n})|^2
\end{align*}
for all $\varepsilon_1,\varepsilon_2>0$ and all $n\geq 1$. Thus, it is enough to obtain the estimates~\eqref{sec40-eq3-1} and~\eqref{sec40-eq3-2} for all $f\in {\mathscr E}$. By approximating we can assume that $f$ is an infinitely differentiable function with compact support. It follows from~\eqref{sec7-eq7.9} that

\begin{align*}
E\bigl|J_{\varepsilon_1}^{-}(H_1,&H_2,f,t)-J_{\varepsilon_2}^{-} (H_1,H_2,f,t)\bigr|^2\\
&=\frac1{\varepsilon_1^{4H_1}}\int_{{\mathbb D}_{0,t}^2} E[f(B^{H_1}_r-\tilde{B}^{H_2}_s)f(B^{H_1}_{r'}-\tilde{B}^{H_2}_{s'})\\
&\hspace{4cm}\cdot(B^{H_1}_{r+\varepsilon_1}-B^{H_1}_r) (B^{H_1}_{r'+\varepsilon_1} -B^{H_1}_{r'})]dr'ds'drds\\
&\qquad-\frac2{\varepsilon_1^{2H_1}\varepsilon_2^{2H_1}}
\int_{{\mathbb D}_{0,t}^2} E[f(B^{H_1}_r-\tilde{B}^{H_2}_s)f(B^{H_1}_{r'}-\tilde{B}^{H_2}_{s'})\\
&\hspace{4cm}\cdot(B^{H_1}_{r+\varepsilon_1}-B^{H_1}_r) (B^{H_1}_{r'+\varepsilon_2} -B^{H_1}_{r'})]dr'ds'drds\\
&\qquad+\frac1{\varepsilon_2^{4H_1}}\int_{{\mathbb D}_{0,t}^2} E[f(B^{H_1}_r-\tilde{B}^{H_2}_s)f(B^{H_1}_{r'}-\tilde{B}^{H_2}_{s'})\\
&\hspace{4cm}\cdot(B^{H_1}_{r+\varepsilon_2}-B^{H_1}_r)(B^{H_1}_{r'+\varepsilon_2} -B^{H_1}_{r'})]dr'ds'drds\\
&=\frac1{\varepsilon_1^{4H_1}\varepsilon_2^{2H_1}}\int_{{\mathbb D}_{0,t}^2}\{\varepsilon_2^{2H_1} \Phi_{\varepsilon_1,\varepsilon_1}-\varepsilon_1^{2H_1}
\Phi_{\varepsilon_1,\varepsilon_2}\} dr'ds'drds\\
&\quad+
\frac1{\varepsilon_1^{2H_1}\varepsilon_2^{4H_1}}\int_{{\mathbb D}_{0,t}^2} \{\varepsilon_1^{2H_1}\Phi_{\varepsilon_2,\varepsilon_2}
-\varepsilon_2^{2H_1} \Phi_{\varepsilon_1,\varepsilon_2}\} dr'ds'drds
\end{align*}
with $\Phi_{\varepsilon_i,\varepsilon_j} :=\Phi_{\varepsilon_i,\varepsilon_j}(r,s,r',s';H_1,H_2)$.
Now, in order to end the proof we need to introduce the following convergence:
\begin{equation}\label{sec7-eq7.12}
\begin{split}
\frac1{\varepsilon_i^{4H_1}\varepsilon_j^{2H_1}} &\int_{{\mathbb D}_{0,t}^2} \{\varepsilon_j^{2H_1} \Phi_{\varepsilon_i,\varepsilon_i}-\varepsilon_i^{2H_1}
\Phi_{\varepsilon_1,\varepsilon_2}\}dr'ds'drds \longrightarrow 0
\end{split}
\end{equation}
with $i,j\in\{1,2\},i\neq j$, as $\varepsilon_1,\varepsilon_2\to 0$. By symmetry, we only need to show that this holds for $i=1,j=2$. Denote
\begin{align*}
A_{0,0}(r,r',\varepsilon,j):&=\varepsilon_j^{2H_1}
E\left[(B^{H_1}_{r+\varepsilon}-B^{H_1}_r) (B^{H_1}_{r'+\varepsilon} -B^{H_1}_{r'})\right]\\
&\qquad-\varepsilon^{2H_1}
E\left[(B^{H_1}_{r+\varepsilon_1}-B^{H_1}_r) (B^{H_1}_{r'+\varepsilon_2} -B^{H_1}_{r'})\right]\\
A_{0,2}(r,r',\varepsilon,j):&=\varepsilon_j^{2H_1} E\left[B^{H_1}_{r'}(B^{H_1}_{r'+\varepsilon}-B^{H_1}_{r'})\right] E\left[B^{H_1}_{r'}(B^{H_1}_{r+\varepsilon}-B^{H_1}_{r})\right]\\
&\qquad-\varepsilon^{2H_1} E\left[B^{H_1}_{r'}(B^{H_1}_{r'+\varepsilon_2}-B^{H_1}_{r'})\right] E\left[B^{H_1}_{r'}(B^{H_1}_{r+\varepsilon_1}-B^{H_1}_{r})\right]\\
A_{2,0}(r,r',\varepsilon,j):&=\varepsilon_j^{2H_1} E\left[B^{H_1}_r(B^{H_1}_{r'+\varepsilon}-B^{H_1}_{r'})\right]
E\left[B^{H_1}_r(B^{H_1}_{r+\varepsilon}-B^{H_1}_{r})\right]\\
&\qquad-\varepsilon^{2H_1} E\left[B^{H_1}_r(B^{H_1}_{r'+\varepsilon_2}-B^{H_1}_{r'})\right]
E\left[B^{H_1}_r(B^{H_1}_{r+\varepsilon_1}-B^{H_1}_{r})\right]\\
A_{1,1}(r,r',\varepsilon,j):&=\varepsilon_j^{2H_1}
E\left[B^{H_1}_{r'}(B^{H_1}_{r'+\varepsilon}-B^{H_1}_{r'})\right] E\left[B^{H_1}_{r}(B^{H_1}_{r+\varepsilon}-B^{H_1}_{r})\right]\\
&\qquad-\varepsilon^{2H_1}
E\left[B^{H_1}_{r'}(B^{H_1}_{r'+\varepsilon_2}-B^{H_1}_{r'})\right] E\left[B^{H_1}_{r}(B^{H_1}_{r+\varepsilon_1}-B^{H_1}_{r})\right]\\
&\qquad+\varepsilon_j^{2H_1}
E\left[B^{H_1}_r(B^{H_1}_{r'+\varepsilon}-B^{H_1}_{r'})\right]
E\left[B^{H_1}_{r'}(B^{H_1}_{r+\varepsilon}-B^{H_1}_{r})\right]\\
&\qquad-\varepsilon^{2H_1} E\left[B^{H_1}_r(B^{H_1}_{r'+\varepsilon_2}-B^{H_1}_{r'})\right]
E\left[B^{H_1}_{r'}(B^{H_1}_{r+\varepsilon_1}-B^{H_1}_{r})\right]
\end{align*}
with $j=1,2$. It follows that
\begin{align*}
\varepsilon_j^{2H_1} \Phi_{\varepsilon_i,\varepsilon_i}-\varepsilon_i^{2H_1}
\Phi_{\varepsilon_1,\varepsilon_2}&=A_{0,0}(r,r',\varepsilon_i,j)E\left[
f(B^{H_1}_r-\tilde{B}^{H_2}_s)f(B^{H_1}_{r'}-\tilde{B}^{H_2}_{s'})\right]\\
&+A_{0,2}(r,r',\varepsilon_i,j)E\left[
f(B^{H_1}_r-\tilde{B}^{H_2}_s)f''(B^{H_1}_{r'}-\tilde{B}^{H_2}_{s'})\right]\\
&+A_{2,0}(r,r',\varepsilon_i,j)
E\left[
f''(B^{H_1}_r-\tilde{B}^{H_2}_s)f(B^{H_1}_{r'}-\tilde{B}^{H_2}_{s'})\right]\\
&
+A_{1,1}(r,r',\varepsilon_i,j)E\left[
f'(B^{H_1}_r-\tilde{B}^{H_2}_s)f'(B^{H_1}_{r'}-\tilde{B}^{H_2}_{s'})\right]
\end{align*}
with $i,j\in \{1,2\}$ and $i\neq j$. Now, it is enough to prove
\begin{equation}\label{sec4-eq4.100}
\int_{{\mathbb D}_{0,t}^2}\frac{A_{k,l}(r,r',\varepsilon_1,2)}{\varepsilon_1^{4H_1} \varepsilon_2^{2H_1}} E\left[
f^{(k)}(B^{H_1}_r-\tilde{B}^{H_2}_s)f^{(l)} (B^{H_1}_{r'}-\tilde{B}^{H_2}_{s'})\right] dr'ds'drds\longrightarrow 0
\end{equation}
for all $k,l\in \{0,1,2\}$ and $k+l\in \{0,2\}$, as $\varepsilon_1,\varepsilon_2\to 0$. Without loss of generality one may assume that $\varepsilon_1>\varepsilon_2$.

For $k=l=0$ and $0<|r-r'|<\varepsilon_2$ we have
$$
|A_{0,0}(r,r',\varepsilon_1,2)|
\leq\varepsilon_2^{2H_1}\varepsilon_1^{2H_1}+\varepsilon_1^{3H_1} \varepsilon_2^{H_1}\leq
\frac{2\varepsilon_2^{2H_1}\varepsilon_1^{\gamma+2H_1}}{|r-r'|^{\gamma}}
$$
by Cauchy's inequality with $2H_1<\gamma\leq 1$. It follows from~\eqref{app1-eq1.4-1} with $\alpha=\frac{\gamma}{2-2H_1}$ that
\begin{align*}
\frac1{\varepsilon_1^{4H_1}\varepsilon_2^{2H_1}} |A_{0,0}(r,r',\varepsilon_1,2)|&\leq  C \left(\frac{1}{|r-r'-\varepsilon_1|^{\gamma}}
+\frac{1}{|r-r'|^{\gamma}}\right)\varepsilon_1^{\gamma-2H_1} \longrightarrow 0
\end{align*}
for all $s,r>0$ and $2H_1<\gamma\leq 1$, as $\varepsilon_1,\varepsilon_2\to 0$. Consequently, Lebesgue's dominated convergence theorem deduces the convergence~\eqref{sec4-eq4.100} with $k=l=0$ because
\begin{align*}
\int_{{\mathbb D}_{0,t}^2}&\left(\frac{1}{|r-r'-\varepsilon_1|^{2H_1}}
+\frac{1}{|r-r'|^{2H_1}}\right)\\
&\qquad\cdot|E\left[
f(B^{H_1}_r-\tilde{B}^{H_2}_s)f(B^{H_1}_{r'}-\tilde{B}^{H_2}_{s'})\right]| dr'ds'drds\leq C \|f\|^2_{\mathscr H}
\end{align*}
for all $\varepsilon_1>0$.

When $k+l=2$, by the fact
\begin{equation}\label{sec6-eq6-18}
b^\alpha-a^\alpha\leq b^{\alpha-\beta}(b-a)^{\beta}
\end{equation}
with $b>a>0$ and $0<\alpha\leq \beta\leq 1$, we have
\begin{equation}\label{sec6-eq6.19}
\begin{split}
|\varepsilon_2^{2H_1}&E[B^{H_1}_{r}(B^{H_1}_{r+\varepsilon_1}-B^{H_1}_r)] -\varepsilon_1^{2H_1}E[B^{H_1}_{r}(B^{H_1}_{r+\varepsilon_2}-B^{H_1}_r)]|\\
&=\left|\frac12\varepsilon_2^{2{H_1}}\left((r+\varepsilon_1)^{2{H_1}} -\varepsilon_1^{2H_1}-r^{2H_1}\right)-\frac12\varepsilon_1^{2H_1} \left((r+\varepsilon_2)^{2H_1}-\varepsilon_2^{2H_1}-r^{2H_1}\right)\right|\\
&=\frac12\left|\varepsilon_2^{2H_1}\left((r+\varepsilon_1)^{2H_1}-r^{2H_1} \right)
-\varepsilon_1^{2H_1}\left((r+\varepsilon_2)^{2H_1}-r^{2H_1}\right)\right|\\
&\leq \frac12\left((r+\varepsilon_1)^{2H_1-\beta}\varepsilon_2^{2H_1} \varepsilon_1^{\beta}+(r+\varepsilon_2)^{2H_1-\beta} \varepsilon_1^{2H_1}\varepsilon_2^{\beta}\right)\leq r^{2H_1-\beta}\varepsilon_2^{2H_1}\varepsilon_1^{\beta}
\end{split}
\end{equation}
for all $2H_1<\beta\leq 1$ and $r>0$. Similarly, by~\eqref{sec6-eq6-18} we also have
\begin{equation}\label{sec6-eq6.20}
\begin{split}
|E\left[B^{H_1}_{r}(B^{H_1}_{r'+\varepsilon}-B^{H_1}_{r'})\right]| &=\frac12\left|
(r'+\varepsilon)^{2H_1}-(r')^{2H_1}+|r-r'|^{2H_1}-|r-r'-\varepsilon|^{2H_1} \right|\\
&\leq \frac12\left((r')^{2H_1-\beta} +|r-r'|^{2H_1-\beta}\right)\varepsilon^{\beta}
\end{split}
\end{equation}
for all $|r-r'|>0,r,r'>0$ and $2H_1<\beta\leq 1$. Combining these with Lemma~\ref{lemA1-3}, we get
\begin{align*}
\frac1{\varepsilon_1^{4H_1}\varepsilon_2^{2H_1}} |A_{0,2}(r,r',\varepsilon_1,2)|
&=\frac1{\varepsilon_1^{4H_1}\varepsilon_2^{2H_1}}\left|E[B^{H_1}_{r'} (B^{H_1}_{r+\varepsilon_1}-B^{H_1}_{r})]\right|\\
&\qquad\cdot\left|\varepsilon_2^{2H_1} E[B^{H_1}_{r'}(B^{H_1}_{r'+\varepsilon_1}-B^{H_1}_{r'})] -\varepsilon^{2H_1}_1 E[B^{H_1}_{r'}(B^{H_1}_{r'+\varepsilon_2}-B^{H_1}_{r'})]\right|\\
&\leq r^{2H_1-\beta}\varepsilon_1^{\beta-2H_1}\longrightarrow 0\quad (\varepsilon_1,\varepsilon_2\to 0),\\
\frac1{\varepsilon_1^{4H_1}\varepsilon_2^{2H_1}} |A_{2,0}(r,r',\varepsilon_1,2)|
&=\frac1{\varepsilon_1^{4H_1}\varepsilon_2^{2H_1}} \left|E[B^{H_1}_r(B^{H_1}_{r+\varepsilon_1}-B^{H_1}_{r})]\right|\\
&\qquad\cdot\left|\varepsilon_2^{2H_1} E[B^{H_1}_r(B^{H_1}_{r'+\varepsilon_1}-B^{H_1}_{r'})] -\varepsilon^{2H_1}_1 E[B^{H_1}_r(B^{H_1}_{r'+\varepsilon_2}-B^{H_1}_{r'})]\right|\\
&\longrightarrow 0\quad (\varepsilon_1,\varepsilon_2\to 0)
\end{align*}
and
\begin{align*}
\frac1{\varepsilon_1^{4H_1}\varepsilon_2^{2H_1}} |A_{1,1}(r,r',\varepsilon_1,2)|
&\leq \frac1{\varepsilon_1^{4H_1}\varepsilon_2^{2H_1}} \left|E[B^{H_1}_{r}(B^{H_1}_{r+\varepsilon_1}-B^{H_1}_{r})]\right|\\ &\qquad\cdot\left|\varepsilon_2^{2H_1}
E[B^{H_1}_{r'}(B^{H_1}_{r'+\varepsilon_1}-B^{H_1}_{r'})] -\varepsilon^{2H_1}_1
E[B^{H_1}_{r'}(B^{H_1}_{r'+\varepsilon_2}-B^{H_1}_{r'})]\right|\\
&\quad+\frac1{\varepsilon_1^{4H_1}\varepsilon_2^{2H_1}} \left|E[B^{H_1}_{r'}(B^{H_1}_{r+\varepsilon_1}-B^{H_1}_{r})]\right|\\
&\qquad\cdot\left|\varepsilon_1^{2H_1}
E[B^{H_1}_r(B^{H_1}_{r'+\varepsilon_1}-B^{H_1}_{r'})] -\varepsilon^{2H_1}_1 E[B^{H_1}_r(B^{H_1}_{r'+\varepsilon_2}-B^{H_1}_{r'})]\right|\\
&\longrightarrow 0\quad (\varepsilon_1,\varepsilon_2\to 0)
\end{align*}
for all $r,r'>0$, which deduce the convergence~\eqref{sec4-eq4.100} for $k+l=2$ by Lebesgue's dominated convergence theorem because
$$
\frac1{\varepsilon_1^{4H_1}\varepsilon_2^{2H_1}} |A_{k,l}(r,r',\varepsilon_1,2)|\leq 2
$$
and
$$
\int_{{\mathbb D}_{0,t}^2} \left|E[f^{(k)}(B^{H_1}_{s})f^{(l)}(B^{H_1}_{r})] \right| dr'ds'drds\leq C \|f\|^2_{\mathscr H}
$$
for $k+l=2$ by Lemma~\ref{lemA1-3}, Lemma~\ref{lem2.5} and Lemma~\ref{lemA1-4}.

Consequently, the convergence~\eqref{sec7-eq7.12} holds for $i=1,j=2$ and~\eqref{sec40-eq3-1} follows. Similarly one can prove~\eqref{sec40-eq3-2}. Thus, we have established the second statement and Theorem~\ref{th7.1} follows.
\end{proof}

At the end of this section, we obtain the Bouleau-Yor type identity~\eqref{sce7-eq7-0011}.

\begin{lemma}\label{lem7.1}
Let $0<H_1<\frac12$ and let $f,f_1,f_2,\ldots\in {\mathscr H}$. If $f_n\to f$ in ${\mathscr
H}$, as $n$ tends to infinity, then we have
$$
[f_n(B^{H_1}-\tilde{B}^{H_2}),B^{H_1}]^{(HC)}_t\longrightarrow [f(B^{H_1}-\tilde{B}^{H_2}),B^{H_1}]^{(HC)}_t
$$
in $L^2$ as $n\to \infty$.
\end{lemma}
\begin{proof}
The convergence follows from
\begin{align*}
E&\left|[f_n(B^{H_1}-\tilde{B}^{H_2}),B^{H_1}]^{(HC)}_t -[f(B^{H_1}-\tilde{B}^{H_2}),B^{H_1}]^{(HC)}_t\right|^2\\
&\hspace{3cm}\leq
C_{H_1,H_2,T}\|f_n-f\|_{\mathscr H}^2\to 0,
\end{align*}
as $n$ tends to infinity.
\end{proof}

\begin{lemma}\label{lem7.2}
Let $0<H_1<\frac12$. For any $f_\triangle=\sum_jf_j1_{(a_{j-1},a_j]}\in {\mathscr E}$, we
define
$$
\int_{\mathbb R}f_\triangle(a)\ell'_t(a)da:=\sum_jf_j\left[\ell_t(a_j) -\ell_t(a_{j-1})\right].
$$
Then the integral is well-defined and
\begin{equation}\label{sec8-eq8.2}
\int_{\mathbb R}f_\triangle(a)\ell'_t(a)da =-[f_\triangle(B^{H_1}-\tilde{B}^{H_2}),B^{H_1}]^{(HC)}_t
\end{equation}
almost surely, for all $t\in [0,T]$.
\end{lemma}
\begin{proof}
For the function $f_\triangle(u)=1_{(x,y]}(u)$ we define the sequence of smooth functions $f_n,\;n=1,2,\ldots$ by
\begin{align}
f_n(u)&=\int_{\mathbb
R}f_\triangle(u-v)\zeta_n(v)dv=\int_x^y\zeta_n(u-v)dv
\end{align}
for all $u\in \mathbb R$, where $\zeta_n,n\geq 1$ are the so-called
mollifiers given by
\begin{equation}\label{sec8-eq8.4}
\zeta_n(u):=n\zeta(nu),\qquad n=1,2,\ldots
\end{equation}
and
\begin{equation}
\zeta(u):=
\begin{cases}
ce^{\frac1{u^2-1}}, &{\text { $-1<u<1$}},\\
0, &{\text { otherwise}}
\end{cases}
\end{equation}
with a normalizing constant $c$ such that $\int_{\mathbb
R}\zeta(u)du=1$. Then $\{f_n\}\subset
C^{\infty}({\mathbb R})\cap {\mathscr H}$ and $f_n$ converges to $f$
in ${\mathscr H}$, as $n$ tends to infinity. It follows from  the occupation formula~\eqref{sec5-eq5.12} that
\begin{align*}
[f_n(B^{H_1}-\tilde{B}^{H_2}),& B^{H_1}]^{(HC)}_t
=\int_0^ts^{2H_2-1}ds\int_0^sf'_n(B^{H_1}_r-\tilde{B}^{H_2}_s)dr\\
&=-\int_{\mathbb R}f_n(a)\ell'_t(a)da=-\int_{\mathbb
R}\left(\int_x^y\zeta_n(a-u)du\right)\ell'_t(a)da\\
&=-\int_x^ydu\int_{\mathbb R}\zeta_n(a-u)\ell'_t(a)da=\int_x^ydu\int_{\mathbb R}\zeta'_n(a-u)\ell_t(a)da\\
&=\int_{\mathbb R}\ell_t(a)da\int_x^y\zeta'_n(a-u)du=-\int_{\mathbb R}\ell_t(a)\left(\zeta_n(a-y)-\zeta_n(a-x)\right)da\\
&=-\int_{\mathbb R}\ell_t(a)\zeta_n(a-y)da
+\int_{\mathbb R}\ell_t(a)\zeta_n(a-x)da\\
&\longrightarrow \ell_t(x)-\ell_t(y)
\end{align*}
almost surely, as $n\to \infty$, by the continuity of $a\mapsto
\ell_t(a)$. On the other hand, Lemma~\ref{lem7.1}
implies that there exists a subsequence $\{f_{n_k}\}$ such that
$$
[f_{n_k}(B^{H_1}-\tilde{B}^{H_2}),B^{H_1}]^{(HC)}_t\longrightarrow [f_\triangle(B^{H_1}-\tilde{B}^{H_2}),B^{H_1}]^{(HC)}_t
$$
for all $t\in [0,T]$, almost surely, as $k\to \infty$, which deduces
$$
[f_\triangle(B^{H_1}-\tilde{B}^{H_2}),B^{H_1}]^{(HC)}_t=\ell_t(x)-\ell_t(y)
$$
for all $t\in [0,T]$, almost surely. Thus, the identity
$$
\int_{\mathbb R}f_\triangle(a)\ell'_t(a)da=-[f_\triangle(B^{H_1}-\tilde{B}^{H_2}), B^{H_1}]^{(HC)}_t
$$
holds and the lemma follows from the linearity property.
\end{proof}

Thanks to the above lemma we can show that
\begin{equation} \lim_{n\to \infty}\int_{\mathbb
R}f_{\triangle,n}(a)\ell'_t(a)da =\lim_{n\to
\infty}\int_{\mathbb R}g_{\triangle,n}(a)\ell'_t(a)da=-[f(B^{H_1}-\tilde{B}^{H_2}), B^{H_1}]^{(HC)}_t
\end{equation}
in $L^2(\Omega)$ if
$$
\lim_{n\to \infty}f_{\triangle,n}(a)=\lim_{n\to
\infty}g_{\triangle,n}(a)=f(a)
$$
in ${\mathscr H}$, where $\{f_{\triangle,n}\},\{g_{\triangle,n}\}\subset
{\mathscr E}$. Thus, by the density of ${\mathscr
E}$ in ${\mathscr H}$ we can define
$$
\int_{\mathbb R}f(a)\ell'_t(a)da:=\lim_{n\to
\infty}\int_{\mathbb R}f_{\triangle,n}(a)\ell'_t(a)da
$$
for any $f\in {\mathscr H}$, where $\{f_{\triangle,n}\}\subset
{\mathscr E}$ and
$$
\lim_{n\to \infty}f_{\triangle,n}(a)=f(a)
$$
in ${\mathscr H}$. Thus, we have proved the following theorem.

\begin{theorem}\label{th7.2}
Let $0<H_1<\frac12$ and $f\in {\mathscr H}$. The integral
$$
\int_{\mathbb R}f(a)\ell'_t(a)da
$$
is well-defined for all $t\in [0,T]$, and the Bouleau-Yor type identity
\begin{align}\label{sec7-eq7.7}
[f(B^{H_1}-\tilde{B}^{H_2}),B^{H_1}]^{(HC)}_t=-\int_{\mathbb R}f(a)\ell'_t(a)da
\end{align}
holds for all $t\in [0,T]$.
\end{theorem}


\section{The hybrid quadratic covariation, case $H_1\geq \frac12$}\label{sec10}
In this section we consider the HQC with $H_2\geq H_1\geq \frac12$ and obtain a similar Bouleau-Yor type identity. It is important to note that the method used in Section~\ref{sec6} is inefficacy for $H_1>\frac12$. Essentially, for $H_1>\frac12$ we have
$$
E\left[(B^{H_1}_{r+\varepsilon}-B^{H_1}_{r}) (B^{H_1}_{r'+\varepsilon}-B^{H_1}_{r'})\right]\sim \varepsilon^2\neq o\left(\varepsilon^{4H_1}\right)\quad (\varepsilon\to 0)
$$
for $r>r'+\varepsilon$ and the decomposition
\begin{align*}
J_\varepsilon(H_1,H_2,t,f)&=\frac1{\varepsilon^{2H_1}} \int_0^t ds\int_0^s f(B^{H_1}_{r+\varepsilon}-\tilde{B}^{H_2}_s) \left(B^{H_1}_{r+\varepsilon}-B^{H_1}_r\right)dr\\
&\qquad\qquad\qquad-\frac1{\varepsilon^{2H_1}}\int_0^t ds\int_0^s f(B^{H_1}_{r}-\tilde{B}^{H_2}_s)\left(B^{H_1}_{r+\varepsilon}-B^{H_1}_r\right)dr\\
&\equiv J^{+}_\varepsilon(H_1,H_2,t,f)-J^{-}_\varepsilon(H_1,H_2,t,f)
\end{align*}
does not bring any information because
$$
EJ^{\pm}_\varepsilon(H_1,H_2,t,f)\longrightarrow \infty\qquad (\varepsilon\to 0)
$$
for $H_1>\frac12$, in general. For example, for $f(x)=x$ we have
\begin{align*}
\frac1{\varepsilon^{2H_1}} \int_0^t ds\int_0^s
&E\left[(B^{H_1}_{r}-\tilde{B}^{H_2}_s)(B^{H_1}_{r+\varepsilon}-B^{H_1}_r)\right]dr\\
&=\frac1{\varepsilon^{2H_1}} \int_0^t ds\int_0^s
E\left[B^{H_1}_{r}(B^{H_1}_{r+\varepsilon}-B^{H_1}_r)\right]dr\\
&\longrightarrow \infty,
\end{align*}
as $\varepsilon\downarrow 0$. Thus, we must estimate $E|J_\varepsilon(H_1,H_2,t,f)|^2$ integrally when $H_1>\frac12$ in order to study the existence of the HQC $[f(B^{H_1}-\tilde{B}^{H_2}),B^{H_1}]^{(HC)}$, and moreover, we shall also use the Young integral
$$
\int_{\mathbb R}f(a)\ell_t(da).
$$

\begin{lemma}\label{lem10.4}
Let $H_1\geq \frac12$.
\begin{itemize}
\item If $\frac1{H_1}+\frac1{H_2}>3$, then for any $f\in {\mathcal C}^{(\nu)}$ with $\nu>0$, the Young integral
$$
\int_{\mathbb R}f(a)\ell_t(da)\equiv\int_{\mathbb R}f(a)\ell'_t(a)da
$$
is well-defined for all $t\geq 0$, and moreover, if $f,f_1,f_2,\ldots\in {\mathcal C }^{\nu}$ and $f_n\to f$ in ${\mathcal C }^{\nu}$, then
\begin{align}\label{sce10-eq10.3}
\int_{\mathbb R}f_n(a)\ell_t(da)\longrightarrow \int_{\mathbb R}f(a)\ell_t(da)
\end{align}
in $L^2(\Omega)$ as $n\to \infty$.
\item If $\frac1{H_1}+\frac1{H_2}\leq 3$, then for $f\in {\mathcal C}^\gamma$ with $\gamma>\frac12\left(3-\frac1{H_1}-\frac1{H_2}\right)$, the Young integral
$$
\int_{\mathbb R}f(a)\ell_t(da)
$$
is well-defined for all $t\geq 0$, and moreover, if $f,f_1,f_2,\ldots\in {\mathcal C }^{\gamma}$ and $f_n\to f$ in ${\mathcal C }^{\gamma}$, then the convergence~\eqref{sce10-eq10.3} holds in $L^2(\Omega)$.
\end{itemize}
\end{lemma}
The above lemma follows from the H\"older continuity of $a\mapsto \ell_t(a)$. For more aspects on Young integration we refer to Dudley-Norvai\v{s}a~\cite{Dudley} and Young~\cite{Young}. Denote
$$
\Upsilon_{\varepsilon}(H_1,H_2):
=E\left[\Delta_{\varepsilon} f(B^{H_1}_{r}-\tilde{B}^{H_2}_{s})\Delta_{\varepsilon} f(B^{H_1}_{r'}-\tilde{B}^{H_2}_{s'}) (B^{H_1}_{r+\varepsilon}-B^{H_1}_r)(B^{H_1}_{r'+\varepsilon} -B^{H_1}_{r'})\right]
$$
for all $\varepsilon>0$, $s,r,s',r'>0$ and Borel functions $f\in {\mathcal C}^\nu$ with $\nu>0$, where
$$
\Delta_\varepsilon f(B^{H_1}_{r}-\tilde{B}^{H_2}_{s})=f(B^{H_1}_{r+\varepsilon}-\tilde{B}^{H_2}_{s}) -f(B^{H_1}_{r}-\tilde{B}^{H_2}_{s}).
$$
Then we have
\begin{align*}
E|J_\varepsilon(H_1,H_2,t,f)|^2=\frac1{\varepsilon^{4H_1}} \int_0^t\int_0^t \int_0^s\int_0^{s'}\Upsilon_{\varepsilon}(H_1,H_2)dr'drds'ds
\end{align*}
for  all $\varepsilon>0$ and $t\in [0,T]$. By approximating we can assume that $f$ is an infinitely differentiable function with compact support. It follows from the duality relationship~\eqref{sec2-eq2.1} and the fact
\begin{equation}\label{d-r-1}
\begin{split}
D^{H_1}_u\Delta_\varepsilon &f(B^{H_1}_{r}-\tilde{B}^{H_2}_{s})=1_{[0,r+\varepsilon]}(u) f'(B^{H_1}_{r+\varepsilon}-\tilde{B}^{H_2}_{s})-1_{[0,r]}(u) f'(B^{H_1}_{r}-\tilde{B}^{H_2}_{s})\\
& =1_{[r,r+\varepsilon]}(u)f'(B^{H_1}_{r+\varepsilon}-\tilde{B}^{H_2}_{s})
+1_{[0,r]}(u) \left\{f'(B^{H_1}_{r+\varepsilon}-\tilde{B}^{H_2}_{s})-f'(B^{H_1}_{r}-\tilde{B}^{H_2}_{s})
\right\}\\
&=1_{[r,r+\varepsilon]}(u)f'(B^{H_1}_{r+\varepsilon}-\tilde{B}^{H_2}_{s})
+1_{[0,r]}(u)\Delta_\varepsilon f'(B^{H_1}_{r}-\tilde{B}^{H_2}_{s})
\end{split}
\end{equation}
that
\begin{align*}
&\Upsilon_{\varepsilon}(H_1,H_2)=E\left[
\Delta_{\varepsilon}f(B^{H_1}_{r}-\tilde{B}^{H_2}_{s})\Delta_{\varepsilon} f(B^{H_1}_{r'}-\tilde{B}^{H_2}_{s'})(B^{H_1}_{r+\varepsilon}-B^{H_1}_r) \int_{r'}^{r'+\varepsilon}dB^{H_1}_l \right]\\
&=E\left[(B^{H_1}_{r+\varepsilon}-B^{H_1}_r) (B^{H_1}_{r'+\varepsilon}-B^{H_1}_{r'})\right]E\left[
f'(B^{H_1}_{r+\varepsilon}-\tilde{B}^{H_2}_s) \Delta_{\varepsilon}f(B^{H_1}_{r'}-\tilde{B}^{H_2}_{s'}) (B^{H_1}_{r+\varepsilon}-B^{H_1}_r)\right]\\
&\quad+E\left[B^{H_1}_r (B^{H_1}_{r'+\varepsilon}-B^{H_1}_{r'})\right]E\left[
\Delta_{\varepsilon}f'(B^{H_1}_{r}-\tilde{B}^{H_2}_s) \Delta_{\varepsilon}f(B^{H_1}_{r'}-\tilde{B}^{H_2}_{s'}) (B^{H_1}_{r+\varepsilon}-B^{H_1}_r)\right]\\
&\quad+E\left[(B^{H_1}_{r'+\varepsilon}-B^{H_1}_{r'})^2\right]E\left[
\Delta_{\varepsilon}f(B^{H_1}_{r}-\tilde{B}^{H_2}_s) f'(B^{H_1}_{r'+\varepsilon}-\tilde{B}^{H_2}_{s'}) (B^{H_1}_{r+\varepsilon}-B^{H_1}_r)\right]\\
&\quad +E\left[B^{H_1}_{r'}(B^{H_1}_{r'+\varepsilon}-B^{H_1}_{r'})\right]E\left[
\Delta_{\varepsilon}f(B^{H_1}_r-\tilde{B}^{H_2}_s) \Delta_{\varepsilon}f'(B^{H_1}_{r'}-\tilde{B}^{H_2}_{s'}) (B^{H_1}_{r+\varepsilon}-B^{H_1}_r)\right]\\
&\quad
+E\left[(B^{H_1}_{r+\varepsilon}-B^{H_1}_r)(B^{H_1}_{r'+\varepsilon} -B^{H_1}_{r'})\right] E\left[\Delta_{\varepsilon}f(B^{H_1}_r-\tilde{B}^{H_2}_s) \Delta_{\varepsilon}f(B^{H_1}_{r'}-\tilde{B}^{H_2}_{s'}) \right]\\
&\equiv \sum_{j=1}^5\Psi_{\varepsilon}(j;H_1,H_2)
\end{align*}
for all $s>r>0,s'>r'>0$ and $\varepsilon>0$. To prove the existence of the HQC, we need to estimate
$$
\Gamma_j:=\frac{1}{\varepsilon^{4H_1}}
\int_0^t\int_0^t dsds'\int_0^s\int_0^{s'}|\Psi_{\varepsilon}(j;H_1,H_2)|dr'dr,\quad j=1,2,3,4,5.
$$
The next lemma is proved in Appendix~\ref{app2}.
\begin{lemma}\label{lemA-1.2}
Let $\frac12<H<1$.
\begin{itemize}
\item [(1)] For all $t>s>r>0$ we have
\begin{equation}\label{sec2-Lemma2-1}
\begin{split}
&\left|E\left[B^H_t(B^H_{t}-B^H_{s})\right]\right|\leq C_Ht^{2H-1}(t-s),\\
&\left|E\left[B^H_t(B^H_{s}-B^H_{r})\right]\right|\leq C_Ht^{2H-1}(s-r),\\
&\left|E\left[B^H_r(B^H_{t}-B^H_{s})\right]\right|\leq C_Hr^{2H-1}(t-s),\\
&\left|E\left[B^H_s(B^H_{t}-B^H_{r})\right]\right|\leq C_Hs^{2H-1}(t-r).
\end{split}
\end{equation}

\item [(2)] For all $0<s'<t'<s<t$ we have
\begin{equation}\label{app1-eq1.3-0}
\left|E\left[(B^H_t-B^H_s)(B^H_{t'}-B^H_{s'})\right]\right|\leq
C_H\frac{(t-s)(t'-s')}{(s-t')^{2-2H}}.
\end{equation}
\end{itemize}
\end{lemma}
\begin{lemma}
Let $\nu>\frac{2H_1-1}{H_1}$ and $f\in {\mathcal C}^\nu({\mathbb R})$. For all $0<\varepsilon<1$ and $t\in [0,T]$, we have
$$
\Gamma_5 \leq C \|f\|^2_{(\nu)}.
$$
\end{lemma}
\begin{proof}
For $j=5$, by~\eqref{app1-eq1.3-0} we have
\begin{align*}
|E[(B^{H_1}_{r+\varepsilon}-B^{H_1}_r)(B^{H_1}_{r'+\varepsilon} -B^{H_1}_{r'})]|\leq C\varepsilon^2 (r-r'-\varepsilon)^{2H_1-2}
\end{align*}
for $r>r'+\varepsilon$, and
\begin{align*}
|E[(B^{H_1}_{r+\varepsilon}-B^{H_1}_r)(B^{H_1}_{r'+\varepsilon} -B^{H_1}_{r'})]|\leq \varepsilon^{2H_1} \leq \varepsilon^2(r-r')^{2H_1-2}
\end{align*}
for $r'<r<r'+\varepsilon$. Combining these with
\begin{align*}
|E[\Delta_{\varepsilon}f(B^{H_1}_r&-\tilde{B}^{H_2}_s) \Delta_{\varepsilon}f(B^{H_1}_{r'}-\tilde{B}^{H_2}_{s'})]|\\
&\leq C\|f\|^2_{(\nu)}E[|B^{H_1}_{r+\varepsilon}-B^{H_1}_r|^\nu |B^{H_1}_{r'+\varepsilon}-B^{H_1}_{r'}|^\nu]|\leq C\|f\|^2_{(\nu)}\varepsilon^{2\nu H_1},
\end{align*}
we get
\begin{align*}
\Gamma_5&\leq C\frac{\varepsilon^{2+2\nu H_1}}{\varepsilon^{4H_1}}
\|f\|^2_{(\nu)}\int_0^t\int_0^t\int_0^s\int_0^{s'} \Bigl[(r-r'-\varepsilon)^{2H_1-2}1_{\{r>r'+\varepsilon\}} \\
&\hspace{6cm}
+(r-r')^{2H_1-2}1_{\{r\leq r'+\varepsilon\}}\Bigr]dr'drds'ds\\
&\leq C \|f\|^2_{(\nu)}
\end{align*}
for all $\nu\geq \frac{2H_1-1}{H_1}$.
\end{proof}
Recall that the notations $\lambda_{r,s},\mu_{\varepsilon_1,\varepsilon_2}$ and $\rho^2_{\varepsilon_1,\varepsilon_2}$ given in Section~\ref{sec3}. Then the next lemma holds which will be proved in Appendix~\ref{app2}.
\begin{lemma}\label{lem2.6}
Let $H_1>\frac12$ and let $s>r>0$, $s'>r'>0$, $s>s'$ and $\varepsilon>0$. Denote
$$
\Lambda_\varepsilon(r,r'):=|r-r'|+ |r+\varepsilon-r'|.
$$
Then we have
\begin{align}\label{sec2-eq2.8}
&\left|\rho_{\varepsilon,\varepsilon}-\rho_{\varepsilon,0}\right|
\leq C\varepsilon \left\{(r+\varepsilon)^{2H_1-1}\vee s^{2H_2-1}\right\}(r'\Lambda_\varepsilon(r,r'))^{-\alpha_1} (s'|s-s'|)^{-\alpha_2},\\ \label{sec2-eq2.9}
&|\frac{\mu_{\varepsilon,0}}{ \rho_{\varepsilon,0}} -\frac{\mu_{\varepsilon,\varepsilon}}{\rho_{\varepsilon,\varepsilon} }|
\leq C\varepsilon \left\{(r+\varepsilon)^{2H_1-1}\vee s^{2H_2-1}\right\}(r'\Lambda_\varepsilon(r,r'))^{-\alpha_1} (s'|s-s'|)^{-\alpha_2},\\  \label{sec2-eq2.10}
&|\frac{\mu_{\varepsilon,0}}{ \rho^2_{\varepsilon,0}} -\frac{\mu_{\varepsilon,\varepsilon}}{\rho^2_{\varepsilon,\varepsilon} }|
\leq C\varepsilon \left\{(r+\varepsilon)^{2H_1-1}\vee s^{2H_2-1}\right\}(r'\Lambda_\varepsilon(r,r'))^{-\alpha_1} (s'|s-s'|)^{-\alpha_2},\\   \label{sec2-eq2.11}
&|\frac{\lambda_{r',s'}}{ \rho^2_{\varepsilon,0}}
-\frac{\lambda_{r'+\varepsilon,s'}}{\rho^2_{\varepsilon,\varepsilon}}|
\leq C\varepsilon (r'\Lambda_\varepsilon(r,r'))^{-\alpha_1} (s'|s-s'|)^{-\alpha_2}
\end{align}
for all $\alpha_1$ and $\alpha_2$ satisfying
\begin{equation}
0<\alpha_1<1-H_1,\quad 0<\alpha_2<1-H_2.
\end{equation}
\end{lemma}

In order to estimate $\Gamma_j$ with $j\geq 2$ we need some preliminaries. Denote
\begin{align*}
\Theta_{r,r'}(i,j)&:=E[
f^{(i)}(B^{H_1}_{r}-\tilde{B}^{H_2}_s) f^{(j)}(B^{H_1}_{r'}-\tilde{B}^{H_2}_{s'})],\\
\Theta_{r,r'}(i,\Delta j)&:=
E[f^{(i)}(B^{H_1}_{r}-\tilde{B}^{H_2}_s) \Delta_{\varepsilon}f^{(j)}(B^{H_1}_{r'}-\tilde{B}^{H_2}_{s'})],\\
\Theta_{r,r'}(\Delta i,\Delta j)&:=
E[\Delta f^{(i)}(B^{H_1}_{r}-\tilde{B}^{H_2}_s) \Delta_{\varepsilon}f^{(j)}(B^{H_1}_{r'}-\tilde{B}^{H_2}_{s'})]
\end{align*}
with $i,j\in\{0,1,2\}$. Let $\varphi_{\varepsilon_1,\varepsilon_2}(x,y)$ be the density function of $
(B^{H_1}_{r+\varepsilon_1}-\tilde{B}^{H_2}_s, B^{H_1}_{r'+\varepsilon_2}-\tilde{B}^{H_2}_{s'})$
with $s>r>0,s'>r'>0$ and $\varepsilon_1,\varepsilon_2\geq 0$. That is
$$
\varphi_{\varepsilon_1,\varepsilon_2}(x,y) =\frac1{2\pi\rho_{\varepsilon_1,\varepsilon_2}} \exp\left\{-\frac{1}{2\rho^2_{\varepsilon_1,\varepsilon_2}}\left(
\lambda_{r'+\varepsilon_2,s'}x^2-2\mu_{\varepsilon_1,\varepsilon_2}xy +\lambda_{r+\varepsilon_1,s}y^2\right)\right\}.
$$

\begin{lemma}\label{lem10.1}
Let $\varepsilon>0,s>r>0$, $s'>r'>0$, $s>s'$ and $f\in C^\infty_0({\mathbb R})\cap {\mathcal C}^\nu$ with $0<\nu\leq 1$, then we have
\begin{align*}
&|\Theta_{r_1,r_2}(i,j)|\leq C \|f\|_{(\nu)}^2 (r\wedge r')^{-H_1}(s')^{-H_2}|r-r'|^{-H_1}|s-s'|^{-H_2} ,\\
&|\Theta_{r,r'}(i,\Delta j)|\leq C\varepsilon^\nu\|f\|_{(\nu)}
\left(rr'\Lambda_\varepsilon(r,r')\right)^{-\alpha_1-H_1} \left(ss'|s-s'|\right)^{-\alpha_2-H_2},\\
&|\Theta_{r,r'}(\Delta i,\Delta j)|\leq C\varepsilon^\nu\|f\|_{(\nu)}
\left(rr'\Lambda_\varepsilon(r,r')\right)^{-\alpha_1-H_1} \left(ss'|s-s'|\right)^{-\alpha_2-H_2}
\end{align*}
for all $0<\alpha_1<1-H_1$ and $0<\alpha_2<1-H_2$, where $i,j\in\{0,1,2\}$ with $i+j=2$ and $r_1,r_2\in\{r+\varepsilon,r'+\varepsilon,r,r'\}$ with $r_1\neq r_2$.
\end{lemma}
\begin{proof}
Let $0<r<s<t$, $0<r'<s'<t$ $s>s'$ and $\varepsilon>0$. We only estimate $\Theta_{r+\varepsilon,r'+\varepsilon}(1,1)$, $\Theta_{r+\varepsilon,r'+\varepsilon}(2,\Delta 0)$, $\Theta_{r+\varepsilon,r'+\varepsilon}(1,\Delta 1)$ and similarly one can estimate the others. We have
\begin{align*}
|\Theta_{r+\varepsilon,r'+\varepsilon}&(1,1)|=\left|\int_{\mathbb{R}^2}
f(x)f(y)\frac{\partial^{2}}{\partial x\partial
y}\varphi_{\varepsilon,\varepsilon}(x,y)dxdy\right|\\
&\leq \int_{\mathbb{R}^2} |f(x)f(y)|\left|\frac1{\rho^4_{\varepsilon,\varepsilon}} (\lambda_{r'+\varepsilon,s'}x-y\mu_{\varepsilon,\varepsilon}
)(\lambda_{r+\varepsilon,s}y-\mu_{\varepsilon,\varepsilon} x)+\frac{\mu_{\varepsilon,\varepsilon}}{\rho^2_{\varepsilon, \varepsilon}}\right|\varphi_{\varepsilon,\varepsilon}(x,y)dxdy\\
&\leq \|f\|_{(\nu)}^2\left(\mu_{\varepsilon,\varepsilon}+
\sqrt{\lambda_{r+\varepsilon,s}\lambda_{r'+\varepsilon,s'}} \right)\frac{1}{\rho^2_{\varepsilon, \varepsilon}}\\
&\leq C \|f\|_{(\nu)}^2 (r\wedge r')^{-H_1}(s')^{-H_2}|r-r'|^{-H_1}|s-s'|^{-H_2}
\end{align*}
by Lemma~\ref{lem2.5}. In order to estimate $\Theta_{r+\varepsilon,r'+\varepsilon}(2,\Delta 0)$, we have
\begin{align*}
E[f''&(B^{H_1}_{r+\varepsilon}-\tilde{B}^{H_2}_s) f(B^{H_1}_{r'+\varepsilon}-\tilde{B}^{H_2}_{s'})]
=\int_{\mathbb{R}^2}
f''(x)f(y)\varphi_{\varepsilon,\varepsilon}(x,y)dxdy\\
&=\int_{\mathbb{R}^2}
f(x)f(y)\frac{\partial^{2}}{\partial x^2}
\varphi_{\varepsilon,\varepsilon}(x,y)dxdy\\
&=\int_{\mathbb{R}^2} f(x)f(y)\left\{\frac1{\rho^4_{\varepsilon,\varepsilon}} (\lambda_{r'+\varepsilon,s'}x-y\mu_{\varepsilon,\varepsilon}
)^2-\frac{\lambda_{r'+\varepsilon,s'}}{\rho^2_{\varepsilon,\varepsilon}} \right\} \varphi_{\varepsilon,\varepsilon}(x,y)dxdy\\
&=\int_{\mathbb{R}^2} f\left(\sqrt{\lambda_{r+\varepsilon,s}}u\right) f\left(\frac{\rho_{\varepsilon,\varepsilon}}{\sqrt{ \lambda_{r+\varepsilon,s}}}v+ \frac{\mu_{\varepsilon,\varepsilon}}{\sqrt{\lambda_{r+\varepsilon,s}}}u \right)\\
&\qquad\cdot\left(
\frac{u^2}{\lambda_{r+\varepsilon,s}} -2\frac{\mu_{\varepsilon,\varepsilon}}{\rho_{\varepsilon,\varepsilon} \lambda_{r+\varepsilon,s}}uv+\frac{\mu_{\varepsilon,\varepsilon}}{
\rho^2_{\varepsilon,\varepsilon}\lambda_{r+\varepsilon,s}}v^2
-\frac{\lambda_{r'+\varepsilon,s'}}{\rho^2_{\varepsilon,\varepsilon}} \right)\frac1{2\pi}e^{-\frac12(u^2+v^2)}dudv
\end{align*}
by making substitutions in the second identity
$$
x=\sqrt{\lambda_{r+\varepsilon,s}}u,\quad y=\frac{\mu_{\varepsilon,\varepsilon}}{\sqrt{\lambda_{r+\varepsilon,s}}}u+ \frac{\rho_{\varepsilon,\varepsilon}}{\sqrt{\lambda_{r+\varepsilon,s}}}v.
$$
It follows that
\begin{align*}
&\Theta_{r+\varepsilon,r'+\varepsilon}(2,\Delta 0)
=E[f''(B^{H_1}_{r+\varepsilon}-\tilde{B}^{H_2}_s) f(B^{H_1}_{r'+\varepsilon}-\tilde{B}^{H_2}_{s'})]-
E[f''(B^{H_1}_{r+\varepsilon}-\tilde{B}^{H_2}_s) f(B^{H_1}_{r'}-\tilde{B}^{H_2}_{s'})]\\
&=\int_{\mathbb{R}^2}f(\sqrt{\lambda_{r+\varepsilon,s}}u)
\frac1{2\pi}e^{-\frac12(u^2+v^2)}dudv\\
&\quad\cdot \left\{f(\frac{\rho_{\varepsilon,\varepsilon}}{ \sqrt{\lambda_{r+\varepsilon,s}}}v+ \frac{\mu_{\varepsilon,\varepsilon}}{ \sqrt{\lambda_{r+\varepsilon,s}}}u)\left(
\frac{u^2}{\lambda_{r+\varepsilon,s}} -\frac{2\mu_{\varepsilon,\varepsilon}}{\rho_{\varepsilon,\varepsilon} \lambda_{r+\varepsilon,s}}uv+\frac{\mu_{\varepsilon,\varepsilon}}{
\rho^2_{\varepsilon,\varepsilon}\lambda_{r+\varepsilon,s}}v^2
-\frac{\lambda_{r'+\varepsilon,s'}}{\rho^2_{\varepsilon,\varepsilon}} \right)\right.\\
&\qquad-\left.
f(\frac{\rho_{\varepsilon,0}}{ \sqrt{\lambda_{r+\varepsilon,s}}}v+ \frac{\mu_{\varepsilon,0}}{ \sqrt{\lambda_{r+\varepsilon,s}}}u)\left(
\frac{u^2}{\lambda_{r+\varepsilon,s}} -\frac{2\mu_{\varepsilon,0}}{\rho_{\varepsilon,0} \lambda_{r+\varepsilon,s}}uv+\frac{\mu_{\varepsilon,0}}{
\rho^2_{\varepsilon,0}\lambda_{r+\varepsilon,s}}v^2
-\frac{\lambda_{r',s'}}{\rho^2_{\varepsilon,0}} \right)\right\}\\
&=\int_{\mathbb{R}^2}f(\sqrt{\lambda_{r+\varepsilon,s}}u)
\left(f(\frac{\rho_{\varepsilon,\varepsilon}}{ \sqrt{\lambda_{r+\varepsilon,s}}}v+ \frac{\mu_{\varepsilon,\varepsilon}}{ \sqrt{\lambda_{r+\varepsilon,s}}}u)-f(\frac{\rho_{\varepsilon,0}}{ \sqrt{\lambda_{r+\varepsilon,s}}}v+ \frac{\mu_{\varepsilon,0}}{ \sqrt{\lambda_{r+\varepsilon,s}}}u)\right)\\
&\qquad\qquad\cdot \left(
\frac{u^2}{\lambda_{r+\varepsilon,s}} -\frac{2\mu_{\varepsilon,\varepsilon}}{\rho_{\varepsilon,\varepsilon} \lambda_{r+\varepsilon,s}}uv+\frac{\mu_{\varepsilon,\varepsilon}}{
\rho^2_{\varepsilon,\varepsilon}\lambda_{r+\varepsilon,s}}v^2
-\frac{\lambda_{r'+\varepsilon,s'}}{\rho^2_{\varepsilon,\varepsilon}} \right)\frac1{2\pi} e^{-\frac12(u^2+v^2)}dudv\\
&\quad+\int_{\mathbb{R}^2}f(\sqrt{\lambda_{r+\varepsilon,s}}u) f(\frac{\rho_{\varepsilon,0}}{ \sqrt{\lambda_{r+\varepsilon,s}}}v+ \frac{\mu_{\varepsilon,0}}{ \sqrt{\lambda_{r+\varepsilon,s}}}u)
\frac1{2\pi}e^{-\frac12(u^2+v^2)}\\
&\qquad\qquad\cdot\left\{\frac{2uv}{\lambda_{r+\varepsilon,s}}
\left(\frac{\mu_{\varepsilon,0}}{\rho_{\varepsilon,0}} -\frac{\mu_{\varepsilon,\varepsilon}}{\rho_{\varepsilon,\varepsilon} }\right)+\frac{v^2}{\lambda_{r+\varepsilon,s}}
\left(\frac{\mu_{\varepsilon,\varepsilon}}{
\rho^2_{\varepsilon,\varepsilon}}
-\frac{\mu_{\varepsilon,0}}{
\rho^2_{\varepsilon,0}}\right)
+\frac{\lambda_{r',s'}}{\rho^2_{\varepsilon,0}}
-\frac{\lambda_{r'+\varepsilon,s'}}{\rho^2_{\varepsilon,\varepsilon}} \right\}dudv\\
&\equiv \Theta_{r+\varepsilon,r'+\varepsilon}(2,\Delta 0,1)+\Theta_{r+\varepsilon,r'+\varepsilon}(2,\Delta 0,2).
\end{align*}
By Lemma~\ref{lem2.6} we have
\begin{align*}
\Theta_{r+\varepsilon,r'+\varepsilon}&(2,\Delta 0,2)\leq C\varepsilon^\nu\|f\|_{(\nu)}(rr'\Lambda_\varepsilon(r,r'))^{ -\alpha_1-H_1}(ss'|s-s'|)^{-\alpha_2-H_2}
\end{align*}
for all $\beta\in (0,1)$. On the other hand, from~\eqref{sec2-eq2.8} and the estimates
\begin{align}\notag
|\mu_{\varepsilon,0}-\mu_{\varepsilon,\varepsilon}|&=
|E[(B^{H_1}_{r+\varepsilon}-\tilde{B}^{H_2}_s) (B^{H_1}_{r'}-\tilde{B}^{H_2}_{s'})] -E[(B^{H_1}_{r+\varepsilon}-\tilde{B}^{H_2}_s) (B^{H_1}_{r'+\varepsilon}-\tilde{B}^{H_2}_{s'})]|\\  \notag
&=|E[(B^{H_1}_{r+\varepsilon}-\tilde{B}^{H_2}_s) (B^{H_1}_{r'+\varepsilon}-B^{H_1}_{r'})]|\\ \notag
&=|E[B^{H_1}_{r+\varepsilon}(B^{H_1}_{r'+\varepsilon} -B^{H_1}_{r'})]|\\ \notag
&=\frac12\left|
(r'+\varepsilon)^{2H_1}-(r')^{2H_1}+|r-r'+\varepsilon|^{2H_1} -|r-r'|^{2H_1}\right|\\ \label{sec2-eq2.7}
&\leq C (r+\varepsilon)^{2H_1-1}\varepsilon,
\end{align}
we find that there are $\alpha_1,\alpha_2$ satisfying
$$
0<\alpha_1<1-H_1,\quad 0<\alpha_2<1-H_2,
$$
such that
$$
\frac{|\mu_{\varepsilon,\varepsilon} -\mu_{\varepsilon,0}|}
{\sqrt{\lambda_{r+\varepsilon,s}}}\leq C\varepsilon (r+\varepsilon)^{-\alpha_1}s^{-\alpha_2}<C\varepsilon r^{-\alpha_1}s^{-\alpha_2}
$$
and
$$
\frac{|\rho_{\varepsilon,\varepsilon} -\rho_{\varepsilon,0}|}
{\sqrt{\lambda_{r+\varepsilon,s}}}\leq C\varepsilon (rr'\Lambda_\varepsilon(r,r'))^{-\alpha_1}(ss'|s-s'|)^{-\alpha_2}
$$
for all $0<\alpha_i<1-H_i$, which imply that
\begin{align*}
|f(\frac{\rho_{\varepsilon,\varepsilon}}{ \sqrt{\lambda_{r+\varepsilon,s}}}v&+ \frac{\mu_{\varepsilon,\varepsilon}}{ \sqrt{\lambda_{r+\varepsilon,s}}}u)-f(\frac{\rho_{\varepsilon,0}}{ \sqrt{\lambda_{r+\varepsilon,s}}}v+ \frac{\mu_{\varepsilon,0}}{ \sqrt{\lambda_{r+\varepsilon,s}}}u)|\\
&\leq C\|f\|_{(\nu)}\left(\frac{|\rho_{\varepsilon,\varepsilon} -\rho_{\varepsilon,0}|^\nu}
{\sqrt{(\lambda_{r+\varepsilon,s})^\nu}}|v|^\nu+ \frac{|\mu_{\varepsilon,\varepsilon}-\mu_{\varepsilon,0}|^\nu}
{\sqrt{(\lambda_{r+\varepsilon,s})^\nu}}|u|^\nu\right)\\
&\leq C\varepsilon^\nu\|f\|_{(\nu)}(rr'\Lambda_\varepsilon(r,r'))^{-\nu\alpha_1}
(ss'|s-s'|)^{-\nu\alpha_2}(|u|+|v|).
\end{align*}
Combining this with
\begin{align*}
Y(u,v):&=|\frac{u^2}{\lambda_{r+\varepsilon,s}} -\frac{2\mu_{\varepsilon,\varepsilon}}{\rho_{\varepsilon,\varepsilon} \lambda_{r+\varepsilon,s}}uv+\frac{\mu_{\varepsilon,\varepsilon}}{
\rho^2_{\varepsilon,\varepsilon}\lambda_{r+\varepsilon,s}}v^2
-\frac{\lambda_{r'+\varepsilon,s'}}{\rho^2_{\varepsilon,\varepsilon}}|\\
&\leq Cr^{-H_1}s^{-H_2}(r')^{-H_1}(s')^{-H_2}|r-r'|^{-H_1}|s-s'|^{-H_2}
\left(u^2+2|uv|+v^2+1\right)
\end{align*}
for all $u,v\in {\mathbb R}$, we get
\begin{align*}
|\Theta_{r+\varepsilon,r'+\varepsilon} (2,\Delta 0,1)|&\leq \int_{\mathbb{R}^2}dudv
|f(\sqrt{\lambda_{r+\varepsilon,s}}u)||Y(u,v)|\frac1{2\pi} e^{-\frac12(u^2+v^2)} \\
&\quad\cdot
|f(\frac{\rho_{\varepsilon,\varepsilon}}{ \sqrt{\lambda_{r+\varepsilon,s}}}v+ \frac{\mu_{\varepsilon,\varepsilon}}{ \sqrt{\lambda_{r+\varepsilon,s}}}u)-f(\frac{\rho_{\varepsilon,0}}{ \sqrt{\lambda_{r+\varepsilon,s}}}v+ \frac{\mu_{\varepsilon,0}}{ \sqrt{\lambda_{r+\varepsilon,s}}}u)|\\
&\leq  C\varepsilon^\nu\|f\|_{(\nu)} (rr'\Lambda_\varepsilon(r,r'))^{-\nu\alpha_1-H_1}
(ss'|s-s'|)^{-\nu\alpha_2-H_2}\\
&\leq  C\varepsilon^\nu\|f\|_{(\nu)}(rr'\Lambda_\varepsilon(r,r'))^{ -\alpha_1-H_1}(ss'|s-s'|)^{-\alpha_2-H_2}
\end{align*}
for all $0<\alpha_i<1-H_i$ since $\nu\in (0,1]$.

Finally, let us estimate $\Theta_{r+\varepsilon,r'+\varepsilon}(1,\Delta 1)$. We have
\begin{align*}
&\Theta_{r+\varepsilon,r'+\varepsilon}(1,\Delta 1)
=\Theta_{r+\varepsilon,r'+\varepsilon}(1,1)-\Theta_{r+\varepsilon,r'}(1,1)\\
&\quad=\int_{\mathbb{R}^2}f(\sqrt{\lambda_{r+\varepsilon,s}}u)
\left\{f(\frac{\rho_{\varepsilon,\varepsilon}}{ \sqrt{\lambda_{r+\varepsilon,s}}}v+\frac{\mu_{\varepsilon,\varepsilon}}{ \sqrt{\lambda_{r+\varepsilon,s}}}u)\left(uv+ \frac{\mu_{\varepsilon,\varepsilon}}{ \rho^2_{\varepsilon,\varepsilon}}(1-v^2)\right)\right.\\
&\qquad-\left.
f(\frac{\rho_{\varepsilon,0}}{ \sqrt{\lambda_{r+\varepsilon,s}}}v+ \frac{\mu_{\varepsilon,0}}{ \sqrt{\lambda_{r+\varepsilon,s}}}u)\left(uv +\frac{\mu_{\varepsilon,0}}{ \rho^2_{\varepsilon,0}}(1-v^2) \right)\right\}\frac1{2\pi}e^{-\frac12(u^2+v^2)} dudv\\
&\quad=\int_{\mathbb{R}^2}f(\sqrt{\lambda_{r+\varepsilon,s}}u)
\left(f(\frac{\rho_{\varepsilon,\varepsilon}}{ \sqrt{\lambda_{r+\varepsilon,s}}}v+ \frac{\mu_{\varepsilon,\varepsilon}}{ \sqrt{\lambda_{r+\varepsilon,s}}}u)-f(\frac{\rho_{\varepsilon,0}}{ \sqrt{\lambda_{r+\varepsilon,s}}}v+ \frac{\mu_{\varepsilon,0}}{ \sqrt{\lambda_{r+\varepsilon,s}}}u)\right)\\
&\hspace{4cm}\cdot\left(uv+\frac{\mu_{\varepsilon,\varepsilon}}{ \rho^2_{\varepsilon,\varepsilon}}(1-v^2)\right)\frac1{2\pi} e^{-\frac12(u^2+v^2)}dudv\\
&+\left(\frac{\mu_{\varepsilon,\varepsilon}}{
\rho^2_{\varepsilon,\varepsilon}}
-\frac{\mu_{\varepsilon,0}}{
\rho^2_{\varepsilon,0}}\right) \int_{\mathbb{R}^2}f(\sqrt{\lambda_{r+\varepsilon,s}}u) f(\frac{\rho_{\varepsilon,0}}{ \sqrt{\lambda_{r+\varepsilon,s}}}v+ \frac{\mu_{\varepsilon,0}}{ \sqrt{\lambda_{r+\varepsilon,s}}}u)
(1-v^2)\frac1{2\pi}e^{-\frac12(u^2+v^2)}dudv
\end{align*}
for all $s>r>0,s'>r'>0$ and $\varepsilon\geq 0$. Thus, the estimate of $\Theta_{r+\varepsilon,r'+\varepsilon}(1,\Delta 1)$ follows from Lemma~\ref{lem2.6} and the estimate of $\Theta_{r+\varepsilon,r'+\varepsilon}(2,\Delta 0)$. This completes the proof.
\end{proof}

\begin{lemma}\label{lem10.3}
Let $\frac{2H_1-1}{H_1}<\nu\leq 1$ and $f\in C^\infty_0({\mathbb R})\cap {\mathcal C}^\nu$. For all $\varepsilon>0$ and $t\in [0,T]$, we have
$$
|\Gamma_1|\leq C \|f\|^2_{(\nu)}.
$$
\end{lemma}
\begin{proof}
Given $\varepsilon> 0$ and $t\in [0,T]$. Denote
$$
h(s,r,s',r'):=E[(B^{H_1}_{s}-B^{H_1}_r) (B^{H_1}_{s'}-B^{H_1}_{r'})]
$$
for all $s>r>0$ and $s'>r'>0$. By the duality relationship~\eqref{sec2-eq2.1} and~\eqref{d-r-1} it follows that
\begin{align*}
\Psi_{\varepsilon}&(1;H_1,H_2)\\
&=h(r+\varepsilon,r,r'+\varepsilon,r')E[
f'(B^{H_1}_{r+\varepsilon}-\tilde{B}^{H_2}_s) \Delta_{\varepsilon}f(B^{H_1}_{r'}-\tilde{B}^{H_2}_{s'}) (B^{H_1}_{r+\varepsilon}-B^{H_1}_r)]\\
&=h(r+\varepsilon,r,r'+\varepsilon,r')
 h(r+\varepsilon,0,r+\varepsilon,r)
 E[f''(B^{H_1}_{r+\varepsilon}-\tilde{B}^{H_2}_s) \Delta_{\varepsilon}f(B^{H_1}_{r'}-\tilde{B}^{H_2}_{s'})]\\
&\qquad+(h(r+\varepsilon,r,r'+\varepsilon,r'))^2 E[
 f'(B^{H_1}_{r+\varepsilon}-\tilde{B}^{H_2}_s) f'(B^{H_1}_{r'+\varepsilon}-\tilde{B}^{H_2}_{s'})]\\
&\qquad +h(r+\varepsilon,r,r'+\varepsilon,r')
 h(r',0,r+\varepsilon,r) E[
 f'(B^{H_1}_{r+\varepsilon}-\tilde{B}^{H_2}_s) \Delta_{\varepsilon}f'(B^{H_1}_{r'}-\tilde{B}^{H_2}_{s'})]\\
&=h(r+\varepsilon,r,r'+\varepsilon,r')
 h(r+\varepsilon,0,r+\varepsilon,r)\Theta_{r+\varepsilon,r'}(2,\Delta 0)\\
&\qquad+(h(r+\varepsilon,r,r'+\varepsilon,r'))^2 \Theta_{r+\varepsilon,r'+\varepsilon}(1,1)\\
&\qquad+h(r+\varepsilon,r,r'+\varepsilon,r')
 h(r',0,r+\varepsilon,r) \Theta_{r+\varepsilon,r'}(1,\Delta 1).
\end{align*}
Combining this with Lemma~\ref{lemA-1.2}, Lemma~\ref{lem10.1} and the fact
$$
|h(s,r,s',r')|=|E[(B^{H_1}_{s}-B^{H_1}_r) (B^{H_1}_{s'}-B^{H_1}_{r'})]|\leq |s-r|^{H_1}|s'-r'|^{H_1},
$$
we get
\begin{align*}
\Gamma_1=\frac{1}{\varepsilon^{4H_1}}
\int_0^t\int_0^t dsds'\int_0^s\int_0^{s'}|\Psi_{\varepsilon}(1;H_1,H_2)|dr'dr<\infty
\end{align*}
if $\nu\geq 2H_1-1$.
\end{proof}

In the same way as Lemma~\ref{lem10.3} one can show that the estimates
\begin{equation}\label{eq301010}
|\Gamma_i|\leq C \|f\|^2_{(\nu)},\quad i=2,3,4
\end{equation}
for all $\varepsilon>0$, $t\in [0,T]$ and $f\in C^\infty_0({\mathbb R})\cap {\mathcal C}^\nu$ with $\nu>2H_1-1$. In fact, by the duality relationship~\eqref{sec2-eq2.1} and~\eqref{d-r-1} we have
\begin{align*}
\Psi_{\varepsilon}(2;H_1,H_2)&=E[B^{H_1}_r (B^{H_1}_{r'+\varepsilon}-B^{H_1}_{r'})]E[
\Delta_{\varepsilon}f'(B^{H_1}_{r}-\tilde{B}^{H_2}_s) \Delta_{\varepsilon}f(B^{H_1}_{r'}-\tilde{B}^{H_2}_{s'}) (B^{H_1}_{r+\varepsilon}-B^{H_1}_r)]\\
&=E[B^{H_1}_r (B^{H_1}_{r'+\varepsilon}-B^{H_1}_{r'})] E[(B^{H_1}_{r+\varepsilon}-B^{H_1}_r)^2]\Theta_{r+\varepsilon,r'}(2,\Delta 0)\\
&\quad+E[B^{H_1}_r (B^{H_1}_{r'+\varepsilon}-B^{H_1}_{r'})]E[B^{H_1}_r (B^{H_1}_{r+\varepsilon}-B^{H_1}_{r})]\Theta_{r,r'}(\Delta 2,\Delta 0)\\
&\quad+E[B^{H_1}_r (B^{H_1}_{r'+\varepsilon}-B^{H_1}_{r'})] \mu_{\varepsilon,\varepsilon}\Theta_{r'+\varepsilon,r}(1,\Delta 1)\\
&\quad +E[B^{H_1}_r (B^{H_1}_{r'+\varepsilon}-B^{H_1}_{r'})] E[B^{H_1}_{r'}(B^{H_1}_{r+\varepsilon}-B^{H_1}_{r})] \Theta_{r,r'}(\Delta 1,\Delta 1),
\end{align*}

\begin{align*}
\Psi_{\varepsilon}(3;H_1,H_2)&=E[(B^{H_1}_{r'+\varepsilon} -B^{H_1}_{r'})^2] E[
\Delta_{\varepsilon}f(B^{H_1}_{r}-\tilde{B}^{H_2}_s) f'(B^{H_1}_{r'+\varepsilon}-\tilde{B}^{H_2}_{s'}) (B^{H_1}_{r+\varepsilon}-B^{H_1}_r)]\\
&=\varepsilon^{4H_1}\Theta_{r+\varepsilon,r'+\varepsilon}(1,1)+
\varepsilon^{2H_1}E[B^{H_1}_{r} (B^{H_1}_{r+\varepsilon}-B^{H_1}_{r})]\Theta_{r'+\varepsilon,r}(1,\Delta 1)\\
&\qquad
+\varepsilon^{2H_1}E[B^{H_1}_{r'+\varepsilon}(B^{H_1}_{r+\varepsilon}-B^{H_1}_{r})]
\Theta_{r'+\varepsilon,r}(2,\Delta 0)
\end{align*}
and
\begin{align*}
\Psi_{\varepsilon}(4;H_1,H_2)&=E[B^{H_1}_{r'}(B^{H_1}_{r'+\varepsilon} -B^{H_1}_{r'})]E[
\Delta_{\varepsilon}f(B^{H_1}_r-\tilde{B}^{H_2}_s) \Delta_{\varepsilon}f'(B^{H_1}_{r'}-\tilde{B}^{H_2}_{s'}) (B^{H_1}_{r+\varepsilon}-B^{H_1}_r)]\\
&=\varepsilon^{2H_1}E[B^{H_1}_{r'}(B^{H_1}_{r'+\varepsilon}-B^{H_1}_{r'})]
\Theta_{r+\varepsilon,r'}(1,\Delta 1)\\
&\qquad+E[B^{H_1}_{r'}(B^{H_1}_{r'+\varepsilon}-B^{H_1}_{r'})] E[B^{H_1}_{r}(B^{H_1}_{r+\varepsilon}-B^{H_1}_{r})] \Theta_{r,r'}(\Delta 1,\Delta 1)\\
&\qquad+h(r+\varepsilon,r,r'+\varepsilon,r') E[B^{H_1}_{r'}(B^{H_1}_{r'+\varepsilon}-B^{H_1}_{r'})] \Theta_{r'+\varepsilon,r}(2,\Delta 0)\\
&\qquad+E[B^{H_1}_{r'}(B^{H_1}_{r'+\varepsilon}-B^{H_1}_{r'})] E[B^{H_1}_{r'}(B^{H_1}_{r+\varepsilon}-B^{H_1}_{r})]\Theta_{r,r'}(\Delta 0,\Delta 2)
\end{align*}
for all $\varepsilon>0$ and $s,r,s',r'>0$. Thus, the estimates~\eqref{eq301010} follows from Lemma~\ref{lemA-1.2} and Lemma~\ref{lem10.1}, and we get the next desired result.
\begin{lemma}\label{lem10.111}
Let $f\in C^\infty_0({\mathbb R})\cap {\mathcal C}^\nu$ with $\nu>\frac{2H_1-1}{H_1}$. Then,
$$
E\left|J_\varepsilon(H_1,H_2,t,f)\right|^2\leq C \|f\|_{(\nu)}^2
$$
for all $\varepsilon>0$ and $t\in [0,T]$.
\end{lemma}

Now, we can obtain our main object of this section.
\begin{theorem}\label{th10.1}
Let $f\in {\mathcal C }^{\nu}$ with $\nu\geq \frac{2H_1-1}{H_1}$ and $\frac1{H_1}+\frac1{H_2}>3$. Then, the HQC $[f(B^{H_1}-\tilde{B}^{H_2}),B^{H_1}]^{(HC)}$ exists, the Bouleau-Yor type identity
\begin{align}\label{sce10-eq10.200}
[f(B^{H_1}-\tilde{B}^{H_2}),B^{H_1}]^{(HC)}_t=-\int_{{\mathbb R}}f(x)\ell'_t(x)dx
\end{align}
and the estimate
\begin{align}\label{sce10-eq10.201}
E\left|[f(B^{H_1}-\tilde{B}^{H_2}),B^{H_1}]^{(HC)}_t\right|^2\leq C \|f\|_{(\nu)}^2
\end{align}
hold for all $t\in [0,T]$.
\end{theorem}
\begin{proof}
Given $f\in {\mathcal C }^{\nu}$. Define the sequence of smooth functions
\begin{align}
f_{n}(x)&=\int_{\mathbb R}f(x-y)\zeta_n(y)dy=
\int_0^2f(x-\frac{y}n)\zeta(y)dy,\qquad n=1,2,\ldots
\end{align}
for all $x\in \mathbb R$, where the mollifiers $\zeta_n,n=1,2,\dots$ are given by~\eqref{sec8-eq8.4}. Then $\{f_{n}\}\subset C^{\infty}_0({\mathbb
R})\cap {\mathcal C }^{\nu}$, $f_{n}$ converges to $f$ in ${\mathcal C }^{\nu}$ and
$$
J_{\varepsilon}(H_1,H_2,t,f_n)\longrightarrow-\int_{{\mathbb R}}f_n(x)\ell'_t(x)dx
$$
in $L^2$ by Corollary~\ref{cor9.1}, as $\varepsilon$ tends to $0$, for all $n\geq 1$.

On the other hand, by Lemma~\ref{lem10.111} we have
\begin{align*}
E\left|J_{\varepsilon}(H_1,H_2,t,f)+\int_{\mathbb R}f(x)\ell'_t(x)dx\right|^2&\leq
3E\left|J_{\varepsilon}(H_1,H_2,t,f)-J_{\varepsilon}(H_1,H_2,t,f_n) \right|^2\\
&\hspace{-6cm}+3E\left|J_{\varepsilon}(H_1,H_2,t,f_n)+\int_{\mathbb R}f_n(x)\ell'_t(x)dx\right|^2 +3E\left|\int_{\mathbb R}f_n(x)\ell'_t(x)dx-\int_{\mathbb R}f(x)\ell'_t(x)dx\right|^2\\
&\hspace{-6cm}\leq 3C  \|f-f_n\|^2_{(\nu)}+3E\left|J_{\varepsilon}(H_1,H_2,t,f_n)+\int_{\mathbb R}f_n(x)\ell'_t(x)dx\right|^2\\
&+3E\left|\int_{\mathbb R}f_n(x)\ell'_t(x)dx-\int_{\mathbb R}f(x)\ell'_t(x)dx\right|^2
\end{align*}
for all $n$, $\varepsilon>0$ and $t\in [0,T]$. Thus, the theorem follows from Lemma~\ref{lem10.4}.
\end{proof}
It is possible to extend formula~\eqref{sce10-eq10.200} to any H\"older functions $f$ of order $\nu>\frac{2H_1-1}{H_1}$ by means of a localization argument. In fact, for any $k\geq 0$ and H\"older functions $f$ of order $\nu>\frac{2H_1-1}{H_1}$ we may consider the set
$$
\Omega_k=\left\{\sup_{0\leq t\leq T}|B^H_t|<k\right\}
$$
and let $f_k$ be a H\"older function such that
$$
f_k(x)=
\begin{cases}
f(-k), & {\text { if $x<-k$}},\\
f(x), & {\text { if $-k\leq x\leq k$}},\\
f(k), & {\text { if $x>k$}}.
\end{cases}
$$
Then $f_k\in {\mathcal C}^\nu$ with $\nu>\frac{2H_1-1}{H_1}$ for every $k\geq 0$. By the above theorem we know that
$$
[f_k(B^{H_1}-\tilde{B}^{H_2}),B^{H_1}]^{(HC)}_t=-\int_{{\mathbb R}}f_k(x)\ell'_t(x)dx
$$
on the set $\Omega_k$. Letting $k$ tend to infinity we get the desired formula~\eqref{sce10-eq10.200} for any H\"older function of order $\nu>\frac{2H_1-1}{H_1}$.

Finally, when $\frac1{H_1}+\frac1{H_2}\leq 3$ we can define
$$
\int_{{\mathbb R}}f(x)\ell_t(dx)=-\int_{{\mathbb R}}f'(x)\ell_t(x)dx
$$
for $f\in C^\infty_0({\mathbb R})\cap {\mathcal C}^\nu$ with
$\nu>\frac12\left(3-\frac1{H_1}-\frac1{H_2}\right)$. It follows from Corollary~\ref{cor9.1}, the occupation formula~\eqref{sec5-eq5.222} and Lemma~\ref{lem10.4} that
\begin{align*}
[f(B^{H_1}-\tilde{B}^{H_2}),B^{H_1}]^{(HC)}&=\int_0^t \int_0^sf'(B^{H_1}_r-\tilde{B}^{H_2}_s)drds\\
&=\int_{{\mathbb R}}f'(x)\ell_t(x)dx=-\int_{{\mathbb R}}f(x)\ell_t(dx)
\end{align*}
for $f\in C^\infty_0({\mathbb R})\cap {\mathcal C}^\nu$ with
$\nu>\frac12\left(3-\frac1{H_1}-\frac1{H_2}\right)$. Thus, by smooth approximating and the localization argument above we get the next result since $\frac{2H_1-1}{H_1}>\frac12\left(3-\frac1{H_1}-\frac1{H_2}\right)$. \begin{theorem}\label{th10.2}
If $f$ is a H\"older function of order $\nu>\frac{2H_1-1}{H_1}$, then the HQC exists and the Bouleau-Yor type identity~\eqref{sce10-eq10.200} holds for all $t\in [0,T]$.
\end{theorem}

\begin{corollary}
Let $B$ and $\tilde{B}$ be two independent Brownian motions and let $f$ be a H\"older function of order $\nu\in (0,1]$. Then, the HQC $[f(B-\tilde{B}),B]^{(HC)}$ and the Young integral
$$
\int_{{\mathbb R}}f(x)\ell_t(dx)=\int_{{\mathbb R}}f(x)\ell'_t(x)dx
$$
exist, and the Bouleau-Yor type identity
\begin{align}
[f(B-\tilde{B}),B]^{(HC)}_t=-\int_{{\mathbb R}}f(x)\ell'_t(x)dx,,
\end{align}
holds for all $t\in [0,T]$, where $\ell'_t(x)$ is the DILT of $B$ and $\tilde{B}$.
\end{corollary}


\section{Appendix: Proofs of some basic estimates}\label{app2}

In this appendix we give proofs of some lemmas.

\begin{proof}[Proof of Lemma~\ref{lemA1-3} and Lemma~\ref{lemA-1.2}]
The inequalities~\eqref{sec2-Lemma1-1} and~\eqref{sec2-Lemma2-1} are some simple exercises. Let us obtain~\eqref{app1-eq1.3} and~\eqref{app1-eq1.3-0}.

For $0<s'<t'<s<t\leq T$ we define the function $x\mapsto G_{s,t}(x)$
on $[s',t']$ by
$$
G_{s,t}(x)=(s-x)^{2H}-(t-x)^{2H}.
$$
Thanks to mean value theorem, we see that there are $\xi\in (s',t')$
and $\eta\in (s,t)$ such that
\begin{align*}
2E\left[(B^H_t-B^H_s)(B^H_{t'}-B^H_{s'})\right]
&=G_{s,t}(t')-G_{s,t}(s')\\
&=2H(t'-s')\left[(t-\xi)^{2H-1}-(s-\xi)^{2H-1}\right]\\
&=2H(2H-1)(t'-s')(t-s)\left(\eta-\xi\right)^{2H-2},
\end{align*}
which gives
\begin{align}\label{app1-eq1.4}
|E\left[(B^H_t-B^H_s)(B^H_{t'}-B^H_{s'})\right]|\leq 2H|2H-1|
\frac{(t'-s')(t-s)}{(s-t')^{2-2H}},
\end{align}
which gives~\eqref{app1-eq1.3-0}. In order to prove~\eqref{app1-eq1.3}, noting that
$$
\frac{|E\left[(B^H_t-B^H_s)(B^H_{t'}-B^H_{s'})\right]|}{
(t-s)^H(t'-s')^H} \leq 1,
$$
we see that
\begin{align*}
&\frac{|E[(B^H_t-B^H_s)(B^H_{t'}-B^H_{s'})]|}{(t-s)^H(t'-s')^H}\leq
\left(\frac{|E\left[(B^H_t-B^H_s)(B^H_{t'}-B^H_{s'})\right]|}{
(t-s)^H(t'-s')^H}\right)^\alpha
\end{align*}
for all $\alpha\in [0,1]$. Combining this with~\eqref{app1-eq1.4} we get
\begin{equation}\label{app1-eq1.4-1}
|E[(B^H_t-B^H_s)(B^H_{t'}-B^H_{s'})]|\leq
\frac{(t-s)^{(1-\alpha)H+\alpha}
(t'-s')^{(1-\alpha)H+\alpha}}{(s-t')^{\alpha(2-2H)}}.
\end{equation}
Since $0<H<\frac12$ we can take $\alpha=H/(1-H)$ and~\eqref{app1-eq1.3} follows.
\end{proof}

\begin{proof}[Proof of Lemma~\ref{lem2.5}]
By symmetry we may assume that $s>s'$. For $0<r'<r<s'<s$, taking $r'=xr$, $s'=ys$ and $0\leq x,y\leq 1$, we have
$$
\lambda_{r',s'}=r^{2H_{1}}x^{2H_{1}}+s^{2H_{2}}y^{2H_2}
$$
and
\begin{align*}
\mu=\frac12r^{2H_1}\left(1+x^{2H_1}-(1-x)^{2H_1}\right)+
\frac12s^{2H_2}\left(1+y^{2H_2}-(1-y)^{2H_2}\right).
\end{align*}
Define the functions :
$$
f_{H}(x):=4x^{2H}-\left(1+x^{2H}-(1-x)^{2H}\right)^{2},
$$
$$
g(x,y):=4\left(x^{2H_{1}}+y^{2H_2}\right) -2\left(1+x^{2H_1}-(1-x)^{2H_1}\right) \left(1+y^{2H_2}-(1-y)^{2H_2}\right)
$$
with $x\in [0,1]$ and $0<H,H_1,H_2<1$. Then
\begin{equation}
\lambda_{r,s}\lambda_{r',s'} -\mu^2=\frac14\left\{r^{4H_{1}}f_{H_1}(x) +r^{2H_{1}}s^{2H_{2}}g(x,y)+s^{4H_{2}}f_{H_2}(y)\right\}.
\end{equation}
It follows from the next lemma that
\begin{align*}
\lambda_{r,s}\lambda_{r',s'} -\mu^2& \asymp r^{4H_{1}}x^{2H_1}(1-x)^{2H_2}+s^{4H_{2}}y^{2H_2}(1-y)^{2H_1}\\
&\qquad\qquad
+r^{2H_{1}}s^{2H_{2}}\left(x^{2H_1}(1-y)^{2H_2}+y^{2H_2}(1-x)^{2H_1} \right)\\
&\asymp \left(r^{2H_{1}} x^{2H_1}+s^{2H_{2}}y^{2H_2}\right)
\left(r^{2H_{1}}(1-x)^{2H_1}+s^{2H_{2}}(1-y)^{2H_2}\right)\\
&\asymp \left((r')^{2H_{1}}+(s')^{2H_{2}}\right)
\left((r-r')^{2H_1}+(s-s')^{2H_2}\right).
\end{align*}
Similarly, we can estimate the $\rho=\lambda_{r,s}\lambda_{r',s'} -\mu^2$ for $0<r<r'<s'<s$ and $0<r'<s'<r<s$, and the lemma follows.
\end{proof}
\begin{lemma}\label{lem3.1}
Let the functions $f_H(x)$ and $g(x,y)$ be defined as above. We then have
\begin{align}\label{sec3-eq3.5}
f_H(x)&\asymp x^{2H}(1-x)^{2H}\\  \label{sec3-eq3.19}
g(x,y)&\asymp x^{2H_1}(1-y)^{2H_2}+y^{2H_2}(1-x)^{2H_1}
\end{align}
for all $x,y\in [0,1]$.
\end{lemma}

The estimates~\eqref{sec3-eq3.5} are given in Yan et al.~\cite{Yan7} (see also Chen-Yan~\cite{Chen-Yan}), and moreover similar to Chen-Yan~\cite{Chen-Yan} we can obtain~\eqref{sec3-eq3.19}.


Now, let us prove the inequalities in Lemma~\ref{lem2.6}. Let $0<r<s<t$, $0<r'<s'<t$, $s>s'$ and $\varepsilon>0$.

\begin{proof}[Proof of~\eqref{sec2-eq2.8}]
Let $r>r'$. Similar to proof of Lemma~\ref{lem2.5}, setting $r'+\varepsilon=x(r+\varepsilon)$, $s'=ys$ and $r'=z(r+\varepsilon)$ we get
\begin{align*}
\rho_{\varepsilon,\varepsilon}^2-\rho_{\varepsilon,0}^2
&=\frac14(r+\varepsilon)^{4H_{1}}\left\{f_{H_1}(x)-f_{H_1}(z)\right\} +\frac14(r+\varepsilon)^{2H_{1}}s^{2H_{2}}\left\{g(x,y)-g(z,y)\right\}\\
&=\frac14(r+\varepsilon)^{4H_{1}}(x-z)f'_{H_1}(\xi) +\frac14(r+\varepsilon)^{2H_{1}}s^{2H_{2}}(x-z)\frac{\partial g}{\partial x}(\eta,y)
\end{align*}
for some $\xi,\eta\in(z,x)$ by the mean value theorem. On the other hand, we have
\begin{align*}
|f'_{H}(x)|&=4H\left|2x^{2H-1}-(x^{2H-1}+(1-x)^{2H-1})(1+x^{2H}-(1-x)^{2H})\right|\\
&=4H\left|2x^{2H-1}-x^{2H-1}(1+x^{2H}-(1-x)^{2H})-(1-x)^{2H-1}
\left(1+x^{2H}-(1-x)^{2H}\right)\right|\\
&\leq 4H\left|2x^{2H-1}-x^{2H-1}(1+x^{2H}-(1-x)^{2H})\right|\\
&\qquad\qquad+4H\left|(1-x)^{2H-1}
\left(1+x^{2H}-(1-x)^{2H}\right)\right|\\
&= 4Hx^{2H-1}\left(1-x^{2H}+(1-x)^{2H}\right)+4H(1-x)^{2H-1}
\left(1+x^{2H}-(1-x)^{2H}\right)\\
&\leq C\left(x^{2H-1}(1-x)+(1-x)^{2H}x\right)\leq C x^{2H-1}(1-x)
\end{align*}
for all $x\in [0,1]$ and $H\in [\frac12,1]$, and
\begin{align*}
|g'_x(x,y)|&=4H |2x^{2H_{1}-1}-\left(x^{2H_1-1}+(1-x)^{2H_1-1}\right) \left(1+y^{2H_2}-(1-y)^{2H_2}\right)|\\
&=4H |2x^{2H_{1}-1}-x^{2H_1-1}\left(1+y^{2H_2}-(1-y)^{2H_2}\right)\\
&\qquad\qquad-4H_1 (1-x)^{2H_1-1}\left(1+y^{2H_2}-(1-y)^{2H_2}\right)|\\
&\leq 4H x^{2H_{1}-1}\left(1-y^{2H_2}+(1-y)^{2H_2}\right)\\
&\qquad\qquad+4H_1 (1-x)^{2H_1-1}\left(1+y^{2H_2}-(1-y)^{2H_2}\right)\\
&\leq C\left(x^{2H_{1}-1}(1-y)+(1-x)^{2H_1-1}y\right)
\end{align*}
for all $x,y\in [0,1]$ and $H_1,H_2\in [\frac12,1]$. It follows that
\begin{align*}
|\rho_{\varepsilon,\varepsilon}^2-\rho_{\varepsilon,0}^2|
&\leq (r+\varepsilon)^{4H_{1}-1}\varepsilon|f'_{H_1}(\xi)| +(r+\varepsilon)^{2H_{1}-1}s^{2H_{2}}\varepsilon|\frac{\partial g}{\partial x}(\eta,y)|\\
&\leq C\xi^{2H-1}(1-\xi)(r+\varepsilon)^{4H_{1}-1}\varepsilon\\ &\qquad+C(r+\varepsilon)^{2H_{1}-1}s^{2H_{2}}\varepsilon \left(\eta^{2H_{1}-1}(1-y)+(1-\eta)^{2H_1-1}y\right)\\
&\leq Cx^{2H-1}(1-z)(r+\varepsilon)^{4H_{1}-1}\varepsilon\\ &\qquad+C(r+\varepsilon)^{2H_{1}-1}s^{2H_{2}}\varepsilon \left(x^{2H_{1}-1}(1-y)+(1-z)^{2H_1-1}y\right)\\
&\leq C\varepsilon\left\{(r+\varepsilon)^{2H_1-1}\vee s^{2H_2-1}\right\}
\left\{(r'+\varepsilon)^{2H_1-1}\vee (s')^{2H_2-1}\right\}\\
&\qquad\cdot\left(|r+\varepsilon-r'|+|r+\varepsilon-r'|^{2H_1-1}+|s-s| \right)\\
&\leq C\varepsilon\left\{(r+\varepsilon)^{2H_1-1}\vee s^{2H_2-1}\right\}\left\{(r'+\varepsilon)^{2H_1-1}\vee (s')^{2H_2-1}\right\}\\
&\qquad\cdot\left(|r+\varepsilon-r'|^{2H_1-1}\vee |s-s|^{2H_2-1} \right)
\end{align*}
for all $r>r'$ and $s>s'$. Thus, for all $r>r'$ and $s>s'$ we have
\begin{align*}
|\rho_{\varepsilon,\varepsilon}-\rho_{\varepsilon,0}|
&=\frac{|\rho_{\varepsilon,\varepsilon}^2-\rho_{\varepsilon,0}^2|}{ \rho_{\varepsilon,\varepsilon}+\rho_{\varepsilon,0}}\\
&\leq \frac{C|\rho_{\varepsilon,\varepsilon}^2-\rho_{\varepsilon,0}^2|}{ \left\{(r'+\varepsilon)^{H_1}+(s')^{H_2}\right\} |s-s'|^{H_2}+(s')^{H_2}\left\{|r+\varepsilon-r'|^{H_1} +|s-s'|^{H_2}\right\}}\\
&\leq \frac{C|\rho_{\varepsilon,\varepsilon}^2-\rho_{\varepsilon,0}^2|}{ (r'+\varepsilon)^{\alpha H_1}(s')^{(1-\alpha)H_2}
|s-s'|^{H_2}+(s')^{H_2}|r+\varepsilon-r'|^{\beta H_1}|s-s'|^{(1-\beta)H_2}}\\
&\leq \frac{C\varepsilon|\rho_{\varepsilon,\varepsilon}^2-\rho_{\varepsilon,0}^2| }{ (r'+\varepsilon)^{\gamma\alpha H_1}(s')^{(\gamma(1-\alpha)+(1-\gamma))H_2} |r+\varepsilon-r'|^{(1-\gamma)\beta H_1}|s-s'|^{((1-\gamma)(1-\beta)+\gamma)H_2}}
\end{align*}
for all $\alpha,\beta,\gamma\in (0,1)$ by Young's inequality. Now, we can take some suitable $\beta_1,\beta_2,\beta_3,\beta_4\in (0,1)$ such that
\begin{align*}
|\rho_{\varepsilon,\varepsilon}-\rho_{\varepsilon,0}|
&\leq \frac{C\varepsilon\left\{(r+\varepsilon)^{2H_1-1}\vee s^{2H_2-1}\right\}}{(r'+\varepsilon)^{\beta_1} (s')^{\beta_2} |r+\varepsilon-r'|^{\beta_3}|s-s'|^{\beta_4}}.
\end{align*}
Moreover, we can also choose the regions of $\beta_1,\beta_2,\beta_3,\beta_4$ as follows
\begin{equation}\label{eq202020}
0<\beta_1, \beta_3<1-H_1,\quad 0<\beta_2, \beta_4<1-H_2.
\end{equation}
In fact, the above conditions can be gotten via comparing the size of $r'+\varepsilon$ and $s'$, and $|r-r'+\varepsilon|$ and $|s-s'|$, respectively. For example, if $(r'+\varepsilon)^{2H_1-1}>(s')^{2H_2-1}$ and $|r-r'+\varepsilon|^{2H_1-1}>|s-s'|$ we then have
\begin{align*}
|\rho_{\varepsilon,\varepsilon}-\rho_{\varepsilon,0}|
&\leq \frac{C\varepsilon\left\{(r+\varepsilon)^{2H_1-1}\vee s^{2H_2-1}\right\}}{(r'+\varepsilon)^{\beta_1}(s')^{\beta_2} |r+\varepsilon-r'|^{\beta_3}|s-s'|^{\beta_4}},
\end{align*}
where
\begin{align*}
&\beta_1:=\gamma\alpha H_1-2H_1+1,\quad\beta_2:= (\gamma(1-\alpha)+(1-\gamma))H_2,\\
&\beta_3:=(1-\gamma)\beta H_1-2H_1+1,\quad\beta_4:=((1-\gamma)(1-\beta)+\gamma)H_2.
\end{align*}
Clearly,
$$
\gamma\alpha H_1-2H_1+1<1-H_1,\quad (1-\gamma)\beta H_1-2H_1+1<1-H_1,
$$
and we can choose $\frac{2H_2-1}{H_2}<\gamma\alpha<1$ and $\frac{2H_1-1}{H_1}<\beta(1-\gamma)<1$ such that~\eqref{eq202020} holds.

Thus, we have obtained~\eqref{sec2-eq2.8} for $r>r'$ and $s>s'$ by taking $\alpha_1=\beta_1\vee \beta_3$ and $\alpha_2=\beta_2\vee \beta_4$. Similarly, we can obtain~\eqref{sec2-eq2.8} for $r<r'$ and $s>s'$.
\end{proof}

\begin{proof}[Proof of~\eqref{sec2-eq2.9}]
By~\eqref{sec2-eq2.8} and Lemma~\ref{lem2.5} we have
\begin{align*}
\frac{\mu_{\varepsilon,0}}{ \rho_{\varepsilon,0}\rho_{\varepsilon,\varepsilon}}&\left|
\rho_{\varepsilon,\varepsilon}-\rho_{\varepsilon,0}\right|\leq C
\frac{((r+\varepsilon)^{H_1}+s^{H_2})}{
(r'+\varepsilon)^{H_1}(s')^{H_2}|r-r'|^{H_1}|s-s'|^{H_2}} \left|\rho_{\varepsilon,\varepsilon}-\rho_{\varepsilon,0}\right|\\
&\leq C\varepsilon \left\{(r+\varepsilon)^{3H_1-1}\vee s^{3H_2-1}\right\}(r'\Lambda_\varepsilon(r,r'))^{-\alpha_1-H_1} (s'|s-s'|)^{-\alpha_2-H_2}
\end{align*}
for all $0<r<s<t$, $0<r'<s'<t$, $0<\alpha_1<1-H_1$ and $0<\alpha_2<1-H_2$. It follows from~\eqref{sec2-eq2.7} that
\begin{align*}
\Bigl|\frac{\mu_{\varepsilon,0}}{ \rho_{\varepsilon,0}} -\frac{\mu_{\varepsilon,\varepsilon}}{\rho_{\varepsilon,\varepsilon} }\Bigr|&=\frac{1}{\rho_{\varepsilon,0} \rho_{\varepsilon,\varepsilon}}\left|
\rho_{\varepsilon,\varepsilon}\mu_{\varepsilon,0}-\rho_{\varepsilon,0}
\mu_{\varepsilon,\varepsilon}\right|\\
&\leq \frac{1}{\rho_{\varepsilon,\varepsilon}}\left|
\mu_{\varepsilon,0}-\mu_{\varepsilon,\varepsilon}\right| +\frac{\mu_{\varepsilon,0}}{ \rho_{\varepsilon,0}\rho_{\varepsilon,\varepsilon}}\left|
\rho_{\varepsilon,\varepsilon}-\rho_{\varepsilon,0}\right|\\
&\leq C\varepsilon \left\{(r+\varepsilon)^{2H_1-1}\vee s^{2H_2-1}\right\}(r'\Lambda_\varepsilon(r,r'))^{-\alpha_1-H_1} (s'|s-s'|)^{-\alpha_2-H_2}
\end{align*}
for all $0<r<s<t$, $0<r'<s'<t$, $0<\alpha_1<1-H_1$ and $0<\alpha_2<1-H_2$.
\end{proof}

\begin{proof}[Proof of~\eqref{sec2-eq2.10} and~\eqref{sec2-eq2.11}]
We have
\begin{equation*}
\begin{split}
\left|\frac{\mu_{\varepsilon,0}}{ \rho^2_{\varepsilon,0}} -\frac{\mu_{\varepsilon,\varepsilon}}{\rho^2_{\varepsilon,\varepsilon} }\right|
&=\frac{1}{\rho^2_{\varepsilon,0} \rho^2_{\varepsilon,\varepsilon}}\left|
\rho^2_{\varepsilon,\varepsilon}\mu_{\varepsilon,0}-\rho^2_{\varepsilon,0}
\mu_{\varepsilon,\varepsilon}\right|\\
&\leq \frac{1}{\rho^2_{\varepsilon,\varepsilon}}\left|
\mu_{\varepsilon,0}-\mu_{\varepsilon,\varepsilon}\right| +\frac{\mu_{\varepsilon,0}}{ \rho^2_{\varepsilon,0}\rho^2_{\varepsilon,\varepsilon}}\left|
\rho^2_{\varepsilon,\varepsilon}-\rho^2_{\varepsilon,0}\right|\\
&\leq C\varepsilon \left\{(r+\varepsilon)^{2H_1-1}\vee s^{2H_2-1}\right\}(r'\Lambda_\varepsilon(r,r'))^{-\alpha_1-H_1} (s'|s-s'|)^{-\alpha_2-H_2}
\end{split}
\end{equation*}
for all $0<r<s<t$, $0<r'<s'<t$, $0<\alpha_1<1-H_1$ and $0<\alpha_2<1-H_2$. Similarly, for~\eqref{sec2-eq2.11} we have
\begin{align*}
\left|\frac{\lambda_{r',s'}}{ \rho^2_{\varepsilon,0}}
-\frac{\lambda_{r'+\varepsilon,s'}}{\rho^2_{\varepsilon,\varepsilon}} \right|&\leq \frac{1}{\rho^2_{\varepsilon,0}}\left|
\lambda_{r',s'}-\lambda_{r'+\varepsilon,s'}\right| +\frac{\lambda_{r'+\varepsilon,s'}}{ \rho^2_{\varepsilon,0}\rho^2_{\varepsilon,\varepsilon}}\left|
\rho^2_{\varepsilon,\varepsilon}-\rho^2_{\varepsilon,0}\right|\\
&\leq C\varepsilon (r'\Lambda_\varepsilon(r,r'))^{-\alpha_1-H_1} (s'|s-s'|)^{-\alpha_2-H_2}
\end{align*}
for all $0<r<s<t$, $0<r'<s'<t$, $0<\alpha_1<1-H_1$ and $0<\alpha_2<1-H_2$.
\end{proof}


\begin{proof}[Proof of Lemma~\ref{lem5.1}]
By making substitutions $u_{j}-u_{j+1}=r_j$, $j=1,2,\ldots,n-1$ and $u_n=r_n$, and then using the estimate
\begin{equation}\label{sec5-eq5.100012}
\int_0^1e^{-x^2u^{2H}}du\asymp \frac{1}{1+|x|^{1/H}},\quad x\in {\mathbb R},\;0<H<1,
\end{equation}
we have
\begin{equation}\label{sec5-eq5.400}
\begin{split}
\Lambda_1(0,t,n,\xi):&=\int_{\widetilde{{\mathbb D}}_{0,t}(u)}
\prod_{j=1}^{n-1}e^{-\frac12 \kappa(\sum\limits_{k=1}^j\xi'_k )^2(u_j-u_{j+1})^{2H_1}}\cdot e^{-\frac12\kappa (\sum\limits_{k=1}^n\xi'_k )^2(u_n)^{2H_1}}du_1\ldots du_n\\
&\leq \prod_{j=1}^{n}\int_0^te^{-\frac12 \kappa(\sum\limits_{k=1}^j\xi'_k )^2r_j^{2H_1}}dr_j\asymp \prod_{j=1}^{n}\bigl(1+|(\sum\limits_{k=1}^j\xi'_k )|^{1/H_1}\bigr)^{-1}
\end{split}
\end{equation}
for all $t\in [0,T]$. On the other hand, some elementary calculus can show that the following estimates hold:
\begin{align*}
\int_{t}^{t'}e^{-x^2u^{2H}}du&=\int_0^{t'-t}e^{-x^2(v+t)^{2H}}dv\leq \int_0^{t'-t}e^{-x^2v^{2H}}dv\\
&=(t'-t)\int_0^1e^{-x^2(t'-t)^{2H}v^{2H}}dv
\asymp \frac{t'-t}{1+(t'-t)|x|^{1/H}}
\end{align*}
with $H\in (0,1)$ and $x\in {\mathbb R}$. By making substitutions $v_{j}-v_{j+1}=s_j$, $j=1,2,\ldots,n-1$ and $v_n=s_n$ it follows that
\begin{equation}\label{sec5-eq5.400-1}
\begin{split}
\Lambda_2(t,t',n,\xi)&:=\int_{\widetilde{{\mathbb D}}_{t,t'}(v)}
\prod_{j=1}^{n-1}e^{-\frac{\kappa}2(\sum\limits_{k=1}^j\xi''_k )^2(v_j-v_{j+1})^{2H_2}} \cdot e^{-\frac{\kappa}2 (\sum\limits_{k=1}^n\xi''_k )^2(v_n)^{2H_2}} dv_j\\
&\leq \prod_{j=1}^n\int_t^{t'}
 e^{-\frac12 \kappa(\sum\limits_{k=1}^j\xi''_k )^2s_j^{2H_2}}ds_j\\
&\leq
C(t'-t)^n\prod\limits_{j=1}^{n}\Bigl(1+(t'-t) |\sum\limits_{k=1}^j\xi''_k |^{1/H_2}\Bigr)^{-1}
\end{split}
\end{equation}
for all $0<H_2<1$. Consequently, we get
\begin{equation}\label{sec5-eq5.400-3}
\begin{split}
\widetilde{\Lambda}(t,t',n,\gamma)&=\int_{\mathbb{R}^n} \Lambda_1(0,t,n,\xi)\Lambda_2(t,t',n,\xi) \prod_{j=1}^n|\xi_j|^{1+\gamma}d\xi_j\\
&\leq \left(\int_{\mathbb{R}^n}
[\Lambda_1(0,t,n,\xi)]^p\prod\limits_{j=1}^n|\xi_j|^{p(1+\gamma)\alpha}d\xi_j \right)^{1/p}\\
&\qquad\qquad\cdot\left(\int_{\mathbb{R}^n}[\Lambda_2(t,t',n,\xi)]^q
\prod\limits_{j=1}^n |\xi_j|^{q(1+\gamma)(1-\alpha)}d\xi_j\right)^{1/q}
\end{split}
\end{equation}
for all $\alpha\in (0,1)$ and $p,q>1,\frac1p+\frac1q=1$.

Finally, making substitutions $\sum\limits_{k=1}^{j}\xi'_k=x_j,\;j=1,2,\ldots,n$, we see that
\begin{align*}
\prod_{j=1}^n|\xi_j|=\prod_{j=1}^n|x_j-x_{j-1}|&\leq \prod_{j=1}^n\left(|x_j|+|x_{j-1}|\right)
\leq \prod_{j=1}^n\left(1+|x_j|\right)\left(1+|x_{j-1}|\right)\\
&\leq \prod_{j=1}^n\left(1+|x_j|\right)^2
\leq 2\prod_{j=1}^n\left(1+|x_j|^2\right)
\end{align*}
with $x_0=0$, and
\begin{equation}\label{sec5-eq5.400-4}
\begin{split}
\int_{\mathbb{R}^n}
[\Lambda_1(0,t,n,\xi)]^p&\prod\limits_{j=1}^n|\xi_j|^{p(1+\gamma)\alpha}d\xi_j
\leq C\int_{\mathbb{R}^n}\prod\limits_{j=1}^n
(1+|\sum\limits_{k=1}^{j}\xi'_k|^{1/H_1})^{-p}|\xi_j|^{p(1+\gamma)\alpha}d\xi_j\\
&\leq C\int_{\mathbb{R}^n}\prod\limits_{j=1}^n
(1+|x_j|^{1/H_1})^{-p}|x_j-x_{j-1}|^{p(1+\gamma)\alpha}dx_j\\
&\leq C\prod\limits_{j=1}^n\int_{\mathbb{R}}
(1+|x_j|^{1/H_1})^{-p}\left(1+|x_j|^2\right)^{p(1+\gamma)\alpha}dx_j
<\infty,
\end{split}
\end{equation}
provided
\begin{equation}\label{sec5-eq5.400-40000}
\frac{p}{H_1}-2p(1+\gamma)\alpha>1.
\end{equation}
Similarly, we have also
\begin{align*}
\int_{\mathbb{R}^n}[\Lambda_2(t,t',n,&\xi)]^q
\prod\limits_{j=1}^n |\xi_j|^{q(1+\gamma)(1-\alpha)}d\xi_j\\
&\leq C(t'-t)^{nq}
\int_{\mathbb{R}^n}\prod\limits_{j=1}^{n}\Bigl(1+(t'-t)|\sum\limits_{k=1}^j\xi''_k |^{1/H_2}\Bigr)^{-q}|\xi_j|^{q(1+\gamma)(1-\alpha)}d\xi_j\\
&\leq C(t'-t)^{nq}\int_{\mathbb{R}^n}\prod\limits_{j=1}^n
(1+(t'-t)|x_j|^{1/H_2})^{-q}|x_j-x_{j-1}|^{q(1+\gamma)(1-\alpha)}dx_j\\
&\leq C(t'-t)^{nq}\prod\limits_{j=1}^n\int_{\mathbb{R}}
(1+(t'-t)|x_j|^{1/H_2})^{-q}\left(1+|x_j|^2\right)^{q(1+\gamma)(1-\alpha)}dx_j\\
&\leq C (t'-t)^{nq-n[1+2q(1+\gamma)(1-\alpha)]H_2},
\end{align*}
provided
\begin{equation}\label{sec5-eq5.400-40001}
\frac{q}{H_2}-2q(1+\gamma)(1-\alpha)>1.
\end{equation}
It follows~\eqref{sec5-eq5.400-3} that
\begin{align*}
\widetilde{\Lambda}(t,t',n,\gamma) \leq C(t'-t)^{n\theta}
\end{align*}
with $\theta\leq 1-H_2\left(\frac1q+2(1+\gamma)(1-\alpha)\right)$, provided
\begin{equation}\label{sec5-eq5.400-4000100}
\begin{cases}
1-H_2\left(\frac1q+2(1+\gamma)(1-\alpha)\right)>0,&\\
\frac{q}{H_2}-2q(1+\gamma)(1-\alpha)>1,&\\
\frac{p}{H_1}-2p(1+\gamma)\alpha>1&.
\end{cases}
\end{equation}
Noting that~\eqref{sec5-eq5.400-4000100} is equivalent to
$$
\begin{cases}
2(1+\gamma)(1-\alpha)<\frac{1}{H_2}-\frac1q,&\\
2(1+\gamma)\alpha<\frac{1}{H_1}-\frac1p,&
\end{cases}
$$
we get $2\gamma<\frac{1}{H_1}+\frac{1}{H_2}-3$, and moreover by taking $p=\frac1{\alpha},q=\frac1{1-\alpha}$ we also have
\begin{align*}
1-H_2&\left(\frac1q+2(1+\gamma)(1-\alpha)\right)
=1-H_2(3+2\gamma)(1-\alpha)\\
&>1-H_2(1-\alpha)\left(\frac{1}{H_1}+\frac{1}{H_2}\right) =\frac1{H_1}\left(\alpha(H_1+H_2)-H_2\right)
\end{align*}
for all $\frac{H_2}{H_1+H_2}<\alpha<1$, which shows that
\begin{align*}
\widetilde{\Lambda}(t,t',n,\gamma) \leq C(t'-t)^{n\beta}
\end{align*}
with $\beta=\frac1{H_1}\left(\alpha(H_1+H_2)-H_2\right)$ and $\frac{H_2}{H_1+H_2}<\alpha<1$. This completes the proof.
\end{proof}



\end{document}